\documentclass[webpdf,modern,mediumone]{oup-authoring-template}
\setcitestyle{numbers,square}
\providecommand{\societylogo}[1]{}
\usepackage{mathtools}

\theoremstyle{thmstyleone}
\newtheorem{theorem}{Theorem}[subsection]
\newtheorem{lemma}[theorem]{Lemma}
\newtheorem{corollary}[theorem]{Corollary}

\theoremstyle{thmstylethree}
\newtheorem{definition}[theorem]{Definition}
\newtheorem{remark}[theorem]{Remark}

\numberwithin{equation}{section}

\begin{document}

\journaltitle{}
\copyrightyear{}
\pubyear{}

\title[Special K\"ahler Metric Singularities on the Hitchin Base]{Local Models for Special K\"ahler Metric Singularities Along the Discriminant Locus of the $\mathrm{SL}_2(\mathbb{C})$ Hitchin Base}

\author[1]{Zhenxi Huang}
\author[2,*]{Shuo Wang}
\author[2]{Bin Xu}

\address[1]{College of Information Science and Technology, Jinan University, Guangzhou 510632, People's Republic of China. Email: \href{mailto:huangzhenxi@jnu.edu.cn}{huangzhenxi@jnu.edu.cn}.}

\address[2]{School of Mathematical Sciences, University of Science and Technology of China, Hefei 230026, People's Republic of China. Emails: \href{mailto:ws220122@mail.ustc.edu.cn}{ws220122@mail.ustc.edu.cn} (S. Wang) and \href{mailto:bxu@ustc.edu.cn}{bxu@ustc.edu.cn} (B. Xu).}

\corresp[*]{Corresponding author: Shuo Wang. Author names are listed in alphabetical order.}

\abstract{Freed formulated the theory of special K\"ahler structures. In particular, the regular part of the $\mathrm{SL}_2(\mathbb C)$ Hitchin base $\mathcal B$ carries such a structure, and the associated metric $\omega_{\mathrm{SK}}$ develops singularities near the discriminant locus $\mathcal D \subset \mathcal B$. Baraglia--Huang showed that its Taylor expansion at regular points can be computed by topological recursion, and Hitchin later constructed integrable subsystems over smooth strata of $\mathcal{D}$ that parametrize spectral curves with only nodal singularities; these strata carry induced special K\"ahler structures. We show that for these strata, the special K\"ahler metric on $\mathcal{B}\setminus\mathcal{D}$ exhibits logarithmic asymptotics in transverse directions, whereas its tangential part converges to a metric on the stratum that agrees with the one from Hitchin's subsystems. We also prove that the induced metric on every radial complex line in these strata is flat with cone angle $\pi$ at the origin.}



\maketitle

\section{Introduction}

In 1987, Hitchin \cite{hitchin1987self,hitchin1987stable} studied stable Higgs bundles over a compact Riemann surface and constructed important examples of algebraically completely integrable systems. Let $C$ be a compact Riemann surface of genus $g\geq 2$, and let $\mathcal M$ denote the moduli space of polystable $\mathrm{SL}_2(\mathbb C)$ Higgs bundles over $C$. Hitchin \cite{hitchin1987self} showed that the stable locus $\mathcal M^{\mathrm{s}}\subset\mathcal M$ is a smooth manifold of complex dimension $6g-6$ and introduced the map
\[
\operatorname{Hit}\colon\mathcal M^{\mathrm{s}}\longrightarrow\mathcal B\coloneqq H^0(C,K_C^2),\qquad(E,\Phi)\longmapsto-\det(\Phi).
\]
This is the \textit{Hitchin map}, $\mathcal B$ is the \textit{Hitchin base}, and the resulting integrable system is the \textit{Hitchin system}. Subsequently, Nitsure \cite{nitsure1991moduli} extended the Hitchin map to a proper morphism $\operatorname{Hit}\colon\mathcal M\to\mathcal B$. The Hitchin base decomposes as $\mathcal B=\mathcal B^{\mathrm{reg}}\sqcup\mathcal D$, where $\mathcal B^{\mathrm{reg}}\coloneqq\{q\in\mathcal B\mid q\text{ has only simple zeros}\}$ is the \textit{regular part}, and $\mathcal D\coloneqq\mathcal B\setminus\mathcal B^{\mathrm{reg}}$ is the \textit{discriminant locus}.

Special K\"ahler geometry appeared in theoretical physics \cite{de1984potentials,gates1984superspace} in 1984, and Freed \cite{freed1999special} gave a rigorous mathematical formulation of special K\"ahler geometry. A special K\"ahler manifold is a K\"ahler manifold $(M,\omega,I)$ equipped with a flat torsion-free symplectic connection $\nabla$ satisfying $\mathrm d_\nabla I=0$, where $\omega$ and $\nabla$ are called the \textit{special K\"ahler metric} and \textit{special K\"ahler connection}, respectively. Moreover, Freed proved that the regular part of the base of an algebraic completely integrable system carries a natural special K\"ahler structure, and so $\mathcal B^{\mathrm{reg}}$ does. We denote the special K\"ahler metric on $\mathcal B^{\mathrm{reg}}$ by $\omega_{\mathrm{SK}}$.

In this paper, we study the asymptotic behavior of the special K\"ahler metric $\omega_{\mathrm{SK}}$ near the discriminant locus $\mathcal D$. The singularities of $\omega_{\mathrm{SK}}$ reflect the degeneration of the Hitchin fibers, since in certain local coordinates its coefficient matrix is given by the imaginary part of the period matrix of these Hitchin fibers. Moreover, the special K\"ahler metric on $\mathcal B^{\mathrm{reg}}$ induces the semi-flat hyperk\"ahler metric on $\mathcal M^{\mathrm{reg}}\coloneqq\operatorname{Hit}^{-1}(\mathcal B^{\mathrm{reg}})$. The semi-flat metric approaches the Hitchin metric with exponential accuracy along certain rays in $\mathcal M$; see \cite{mazzeo2019asymptotic,DN2019,Fredrickson2020,he2025asymptotics} for related results. It is the starting point of the Gaiotto--Moore--Neitzke (GMN) construction \cite{GMN2010,gaiotto2013wall}. More recently, Fredrickson--Zimet \cite{fredrickson2025hyper} constructed multi-Ooguri--Vafa-type model geometries near the discriminant locus in the Coulomb branch under suitable assumptions on the singularities, providing a framework for extending the GMN formalism beyond the regular part. These developments motivate understanding the singular behavior of $\omega_{\mathrm{SK}}$.

For $q\in\mathcal B$, let $\Sigma_q$ denote the associated spectral curve. It is smooth for $q\in\mathcal B^{\mathrm{reg}}$, and has only nodal singularities for $q\in\bigcup_{1\leq d\leq 2g-3}\mathcal B_d$, $q\in\mathcal B_{2g-2}$, or $q\in\mathcal B_{\mathrm{ab}}$, where
\begin{align*}
\mathcal B_d&\coloneqq\left\{q\in\mathcal B\;\middle|\;q\text{ has }d\text{ double zeros and }4g-4-2d\text{ simple zeros}\right\},\quad 1\leq d\leq 2g-3,\\
\mathcal B_{2g-2}&\coloneqq\left\{q\in\mathcal B\;\middle|\;q\text{ has }2g-2\text{ double zeros and no global square root}\right\},\\
\mathcal B_{\mathrm{ab}}&\coloneqq\left\{q=\psi^2\;\middle|\;\psi\in H^0(C,K_C)\text{ has only simple zeros}\right\}.
\end{align*}
Hitchin \cite{hitchin2019critical,hitchin2021integrable} constructed integrable systems over the strata $\mathcal B_d$ for $1\leq d\leq2g-3$ and over $\mathcal B_{2g-2}$. These systems are embedded in the $\mathrm{SL}_2(\mathbb C)$ Hitchin system and are therefore called \textit{integrable subsystems}. Analogously, the moduli space of split Higgs bundles $\left(E=L\oplus L^{-1},\Phi=\operatorname{diag}(\psi,-\psi)\right)$ gives an integrable system over $\mathcal B_{\mathrm{ab}}$. By the general theory \cite{freed1999special}, each stratum $\mathcal B_d$, $\mathcal B_{2g-2}$, and $\mathcal B_{\mathrm{ab}}$ carries a special K\"ahler structure, and we denote the corresponding special K\"ahler metrics by $\omega_{\mathrm{SK},\mathcal B_d}$, $\omega_{\mathrm{SK},\mathcal B_{2g-2}}$, and $\omega_{\mathrm{SK},\mathcal B_{\mathrm{ab}}}$, respectively.

\subsection{Main results}\label{subsec:intro-main}

Here and below, for functions $f$ and $g$ on a common domain, the notation $f\asymp g$ means that there exist constants $c,C>0$ such that $c\,g(q)\leq f(q)\leq C\,g(q)$ for all $q$ in their common domain. 

Let $\mathcal S$ be one of the strata $\mathcal B_d$ ($1\leq d\leq2g-3$), $\mathcal B_{2g-2}$, or $\mathcal B_{\mathrm{ab}}$, and define $m$ accordingly by $m=3g-3-d$, $g-1$, or $g$, respectively. Fix $q_0\in\mathcal S$. We now state our main results.

\begin{theorem}[Main asymptotic theorem]\label{thm:main-asymptotic}
There exists a neighborhood $\mathcal U\subset\mathcal B$ of $q_0$ and holomorphic coordinates $\{\mathfrak z^i\}_{i=1}^{3g-3}$ on $\mathcal U$ such that $\mathcal S\cap\mathcal U=\{\mathfrak z^{m+1}=\cdots=\mathfrak z^{3g-3}=0\}$ is a complex submanifold of $\mathcal U$ of dimension $m$. The special K\"ahler metric $\omega_{\mathrm{SK}}$ is expressed as
\[
\omega_{\mathrm{SK}}(q)=\frac{\mathrm i}{2}\sum_{i,j=1}^{3g-3}\omega_{ij}(q)\,\mathrm d\mathfrak z^i\wedge\mathrm d\overline{\mathfrak z}^j,\quad q\in\mathcal U\cap\mathcal B^{\mathrm{reg}}.\quad \text{ Denote }\left(\omega_{ij}\right)=
\begin{pmatrix}
A&B\\
B^{\mathsf T}&D
\end{pmatrix}.
\]
Here $A$ and $D$ have orders $m$ and $3g-3-m$, respectively. As $\mathcal B^{\mathrm{reg}}\ni q\to q_0$, the following hold.

\begin{enumerate}
    \item The blocks $A(q)$ and $B(q)$ extend continuously to $q_0$, and $A(q_0)$ is positive definite.

    \item Let $E_q\coloneqq\bigcap_{k=m+1}^{3g-3}\ker\left(\mathrm d\mathfrak z^k|_q\right)\subset T_q\mathcal B$. Then $E_q=T_q\mathcal S$ on $\mathcal S\cap\mathcal U$, and the tangential limit of $\omega_{\mathrm{SK}}$ agrees with the special K\"ahler metric $\omega_{\mathrm{SK},\mathcal S}$ on the stratum $\mathcal S$, i.e.
\[
\lim_{\mathcal B^{\mathrm{reg}}\ni q\to q_0}\left(\omega_{\mathrm{SK}}(q)\big|_{E_q}\right)=\omega_{\mathrm{SK},\mathcal S}(q_0)=\frac{\mathrm i}{2}\sum_{i,j=1}^{m}A_{ij}(q_0)\,\mathrm d\mathfrak z^i|_{q_0}\wedge\mathrm d\overline{\mathfrak z}^j|_{q_0},
\]
where $\omega_{\mathrm{SK},\mathcal S}$ is $\omega_{\mathrm{SK},\mathcal B_d}$, $\omega_{\mathrm{SK},\mathcal B_{2g-2}}$, or $\omega_{\mathrm{SK},\mathcal B_{\mathrm{ab}}}$ according as $\mathcal S=\mathcal B_d$, $\mathcal B_{2g-2}$, or $\mathcal B_{\mathrm{ab}}$, respectively.

    \item Suppose that $\mathcal S=\mathcal B_d$ ($1\leq d\leq 2g-3$) or $\mathcal S=\mathcal B_{2g-2}$. Then every off-diagonal entry of $D(q)$ extends continuously to $q_0$, whereas its diagonal entries diverge logarithmically:
\[
D_{kk}(q)\asymp -\log\left|\mathfrak z^{m+k}(q)\right|,\qquad 1\leq k\leq 3g-3-m.
\]

    \item Suppose that $\mathcal S=\mathcal B_{\mathrm{ab}}$. The logarithmically divergent part of $D(q)$ appears only on its diagonal and first off-diagonal entries, while $D_{ij}(q)$ extends continuously to $q_0$ whenever $|i-j|>1$. The two-sided logarithmic estimates are given in Theorem~\ref{thm:Bab-singular-model}.
\end{enumerate}
\end{theorem}

For comparison, Tosatti--Zhang \cite[Theorem~3.4]{TZ2020} established logarithmic control of limiting special K\"ahler metrics in the setting of collapsing hyperK\"ahler Lagrangian fibrations. At the period-matrix level, Yamada \cite{yamada1980precise} derived variational formulas describing logarithmic behavior under nodal degeneration, while Schmid \cite{schmid1973variation} provided a conceptual framework for such singular behavior through the nilpotent-orbit description of degenerating variations of Hodge structure. For the $\mathrm{SL}_2(\mathbb C)$ Hitchin system, we construct conjugate special coordinates adapted simultaneously to all vanishing cycles of the corresponding multi-nodal degeneration for each stratum $\mathcal B_d$, $\mathcal B_{2g-2}$, and $\mathcal B_{\mathrm{ab}}$. The same construction describes nearby regular spectral curves and their degeneration to the stratum: integrals along vanishing cycles give local equations for the stratum, while the period matrix describes transverse singularities and its tangential block converges to the intrinsic special K\"ahler metric on the stratum. In the split case $\mathcal B_{\mathrm{ab}}$, the transverse singularity is coupled rather than a direct sum of independent one-node models.

We also consider the asymptotic behavior of the special K\"ahler metric along complex lines through the origin. For a complex line $\mathcal L_{q_0}\coloneqq\{\ell q_0\mid\ell\in\mathbb C\}$, let $\omega_{\mathcal L_{q_0}}$ denote the restriction of $\omega_{\mathrm{SK},\mathcal S}$ to $\mathcal L_{q_0}\setminus\{0\}$. The radial behavior of this intrinsic metric is described as follows.

\begin{corollary}[Flat cone metric]\label{thm:intro-radial}
The metric $\omega_{\mathcal L_{q_0}}$ is flat with a cone singularity of angle $\pi$ at the origin:
\[
\omega_{\mathcal L_{q_0}} = \frac{\mathrm{i}}{2} C(q_0) |\ell|^{-1} \,\mathrm d\ell\wedge \mathrm d\bar{\ell},\quad\text{where }C(q_0) = \frac{\mathrm{i}}{4} \int_C \sqrt{q_0 \bar{q}_0}.
\]
\end{corollary}

The function $2C(q_0)|\ell|$ is a K\"ahler potential for $\omega_{\mathcal L_{q_0}}$. When $g=2$, this radial metric agrees with the explicit computation of Hitchin \cite[Section~5.2]{hitchin2021integrable}.

\subsection{Strategy of the proof}\label{subsec:intro-strategy}

We briefly describe the common mechanism underlying our proofs. When $q_0\in\mathcal B^{\mathrm{reg}}$, Baraglia--Huang \cite[Section~2]{baraglia2017special} and Mazzeo--Swoboda--Weiss--Witt \cite[Section~2.3]{mazzeo2019asymptotic} constructed conjugate coordinates on a small simply connected neighborhood of $q_0$ and derived the corresponding coordinate expression for $\omega_{\mathrm{SK}}$. Recall that conjugate coordinates refer to flat Darboux coordinates on a special K\"ahler manifold. In this paper, we adapt this construction to degenerations near $\mathcal B_d$, $\mathcal B_{2g-2}$, and $\mathcal B_{\mathrm{ab}}$, and describe the singular models of $\omega_{\mathrm{SK}}$ via its coefficient matrix.

Fix notation. For each $q\in\mathcal B$, let $\pi_q\colon\Sigma_q\to C$ be the spectral cover, and let $\nu_q\colon\widetilde\Sigma_q\to\Sigma_q$ be the normalization. Then $\widetilde\pi_q\coloneqq\pi_q\circ\nu_q\colon\widetilde{\Sigma}_q\to C$ is a (branched) double cover, which distinguishes three cases: branched for $q\in\mathcal B_d$ ($1\leq d\leq2g-3$), nontrivial unbranched for $q\in\mathcal B_{2g-2}$, and split for $q\in\mathcal B_{\mathrm{ab}}$. Let $\sigma_q$ be the sheet-exchange involution on $\widetilde\Sigma_q$, and let $H_1\big(\widetilde{\Sigma}_q,\mathbb R\big)^\pm$ and $H^{1,0}\big(\widetilde{\Sigma}_q\big)^\pm$ denote the $(\pm1)$-eigenspaces of $\sigma_{q*}$ and $\sigma_q^*$, respectively.

Fix $q_0\in\mathcal B_d$ ($1\leq d\leq2g-3$), $\mathcal B_{2g-2}$, or $\mathcal B_{\mathrm{ab}}$, and let $\mathcal U\subset\mathcal B$ be a small simply connected neighborhood of $q_0$.

\begin{enumerate} 
    \item[Step 1] The key step is to construct $3g-3$ pairs of paths $\left\{\mathfrak a_{i,q_0},\mathfrak b_{i,q_0}\right\}_{i=1}^{3g-3}$ on $\widetilde\Sigma_{q_0}$ and deform them continuously to nearby spectral curves $\widetilde\Sigma_q$ for $q\in\mathcal U$. Denote the resulting paths by $\left\{\mathfrak a_{i,q},\mathfrak b_{i,q}\right\}_{i=1}^{3g-3}$, and require additionally that they form a symplectic basis of $H_1(\Sigma_q,\mathbb R)^-$ for each $q\in\mathcal U\cap\mathcal B^{\mathrm{reg}}$. Note that
\begin{itemize}
    \item Since $\dim H_1\big(\widetilde\Sigma_{q_0},\mathbb R\big)^-<6g-6$, we start with a symplectic basis of $H_1\big(\widetilde\Sigma_{q_0},\mathbb R\big)^-$ and supplement it with additional contractible loops or paths joining pairs of points over the nodes of $\Sigma_{q_0}$.

    \item The deformation of paths may be non-unique at the homology level. Both families $\left\{\mathfrak b_{i,q}\right\}_{q\in\mathcal U}$ and $\left\{\mathfrak b'_{i,q}\right\}_{q\in\mathcal U}$ may be continuous deformations of the same path $\mathfrak b_{i,q_0}$, while $\left[\mathfrak b_{i,q}\right]-\left[\mathfrak b'_{i,q}\right]$ may not be null-homologous.
\end{itemize}

    \item[Step 2] Define the integral functions on $\mathcal U$ as
\[
\mathfrak z^i(q)\coloneqq\int_{\mathfrak a_{i,q}}\vartheta_q,\qquad \mathfrak w_i(q)\coloneqq-\int_{\mathfrak b_{i,q}}\vartheta_q,\qquad 1\leq i\leq3g-3,
\]
where $\vartheta_q$ denotes the Seiberg--Witten form on the spectral curve. Since the homology class of the integration path from Step 1 is not unique, the functions $\{\mathfrak w_i\}_{i=1}^{3g-3}$ may be multi-valued. On simply connected domains in $\mathcal U\cap\mathcal B^{\mathrm{reg}}$, we may choose single-valued branches so that $\left\{\mathfrak z^i,\mathfrak w_i\right\}_{i=1}^{3g-3}$ form conjugate coordinates locally.

    \item[Step 3] Finally, the special K\"ahler metric $\omega_{\mathrm{SK}}$ is given by
\[
\omega_{\mathrm{SK}} = \frac{\mathrm{i}}{2} \sum_{i,j=1}^{3g-3} \operatorname{Im}(\tau_{ij}) \, \mathrm{d} \mathfrak{z}^i \wedge \mathrm{d} \bar{\mathfrak{z}}^j, \quad \text{where} \quad \tau_{ij}(q) = \int_{\mathfrak b_{j,q}} \omega_{i,q}, \quad q\in\mathcal U\cap\mathcal B^{\mathrm{reg}}.
\]
Here $\{\omega_{i,q}\}_{i=1}^{3g-3}$ are differentials on each $\widetilde\Sigma_q$, normalized by $\int_{\mathfrak a_{j,q}}\omega_{i,q}=\delta_{ij}$. In particular, on $\widetilde\Sigma_{q_0}$, some of the differentials $\{\omega_{i,q_0}\}_{i=1}^{3g-3}$ are meromorphic with prescribed simple poles at the preimages of the nodes. Roughly speaking, the integrals of these meromorphic differentials produce the logarithmic divergences of $\omega_{\mathrm{SK}}$.
\end{enumerate}

The remainder of this paper is organized as follows. Section~2 collects the relevant geometry on normalized spectral curves and establishes the local estimates used in the subsequent sections. Sections~3--5 treat the strata $\mathcal B_d$ for $1\leq d\leq 2g-3$, $\mathcal B_{2g-2}$, and $\mathcal B_{\mathrm{ab}}$, respectively. In each case, we construct adapted paths on the normalized spectral curve, deform them, and ultimately derive the corresponding singular model for $\omega_{\mathrm{SK}}$. Section~6 identifies the tangential limit of $\omega_{\mathrm{SK}}$ with the special K\"ahler metric on each stratum and studies its restriction to complex lines through the origin, proving that the latter is a flat metric with cone angle $\pi$.

\section{Normalized Spectral Curves and Local Asymptotics}

This section develops the geometric and analytic tools used in the study of the special K\"ahler metric near the discriminant locus. We describe the normalization of singular spectral curves, identify quadratic differentials on $C$ with anti-invariant meromorphic differentials on the normalized spectral curve, and establish two local estimates governing the behavior of the relevant integrals when two simple zeros collide.

Let $\pi\colon K_C\to C$ denote the projection of the canonical bundle, and let $\vartheta$ be the tautological section of $\pi^*K_C$, equivalently, the tautological $1$-form on $\operatorname{Tot}(K_C)$. For each $q\in\mathcal B$, define the associated spectral curve $\Sigma_q\coloneqq\operatorname{div}(s_q)\subset\operatorname{Tot}(K_C)$, where $s_q\coloneqq\vartheta^2-\pi^*q\in H^0(\operatorname{Tot}(K_C),\pi^*K_C^2)$. We adopt the convention that $(U,z)$ is a coordinate chart on $C$, and $(\pi^{-1}(U),\widetilde{x},\widetilde{z}\coloneqq z\circ\pi)$ is the induced chart on $\operatorname{Tot}(K_C)$, so that locally $\vartheta=\widetilde{x}\,\mathrm d\widetilde{z}$.

\subsection{Geometry of normalized spectral curves}\label{subsec:geometry-spectral-curves}

Fix a nonzero $q\in\mathcal B$ in this subsection. View $\widetilde{\Sigma}_q$ as lying in $\operatorname{Tot}(K_C)$ through $\widetilde{\Sigma}_q\to\Sigma_q\hookrightarrow \operatorname{Tot}(K_C)$, though it is not usually an embedding, and denote by $\vartheta_q$ the pullback of $\vartheta$ to $\widetilde{\Sigma}_q$. This $\vartheta_q$ is called the Seiberg--Witten differential in the literature. We give its coordinate description as follows. 

Let $(U,z)$ be a chart centered at a point $p\in C$, and shrink $U$ if necessary. If $p$ is not a zero of $q$ or is a zero of even order, write $q=z^{2n}\,\mathrm d z^2$; then $\widetilde{\pi}_q^{-1}(U)$ has two components $\widetilde U_q^\pm$. We may take a coordinate $t$ on $\widetilde U_q^\pm$ such that the map $\widetilde\Sigma_q\to\operatorname{Tot}(K_C)$ is locally given by $t \mapsto (\widetilde{x}=\pm t^n,\widetilde z=t)$. If $p$ is a zero of odd order, write $q=z^{2n+1}\,\mathrm d z^2$; then $\widetilde{\pi}_q^{-1}(U)$ has a single component $\widetilde U_q$ equipped with a coordinate $t$. In this case, $\widetilde{\Sigma}_q$ is given by $t\mapsto(\widetilde{x}=t^{2n+1},\widetilde z=t^2)$. Then $\vartheta_q = \widetilde{x} \,\mathrm{d}\widetilde{z}$ has the local expressions:
\[
\vartheta_q =
\begin{cases}
\pm t^n \,\mathrm d t & \text{on } \widetilde U_q^\pm \text{ if } q = z^{2n} \,\mathrm d z^2 \text{ on } U,\\
2t^{2n+2} \,\mathrm d t & \text{on } \widetilde U_q \text{ if } q = z^{2n+1} \,\mathrm d z^2 \text{ on } U.
\end{cases}
\]

\begin{lemma}\label{lem:canonical-bundle-isomorphism}
Let $D_q\coloneqq\left\lfloor\frac{1}{2}\operatorname{div}(q)\right\rfloor$ be the floor of the half-divisor. There is a bundle isomorphism with local expression as follows:
\begin{equation}\label{eq:canonical-bundle-map}
\Psi_q\colon \widetilde\pi_q^*\bigl(K_C^2(-D_q)\bigr) \xrightarrow{\;\cong\;} K_{\widetilde\Sigma_q},\qquad \widetilde\pi_q^*(\mathrm d z^2) \mapsto \frac{\mathrm d\widetilde z^2}{2\vartheta_q}.
\end{equation}
\end{lemma}

\begin{proof}
One can check that the local expression \eqref{eq:canonical-bundle-map} defines an isomorphism away from $\widetilde\pi_q^{-1}(\operatorname{Zero}(q))$. Let $(U,z)$ be centered at a zero of $q$ of even order, and write $q=z^{2n}\,\mathrm d z^2$. Then $z^n\,\mathrm d z^2$ is a local frame of $K_C^2(-D_q)$. Recall that $\widetilde x=\pm t^n$ and $\widetilde z=t$ on each chart $(\widetilde U_p^\pm,t)$ on $\widetilde\Sigma_q$. Therefore,
\[
\Psi_q\bigl(\widetilde\pi_q^*(z^n\,\mathrm d z^2)\bigr)=\frac{t^n\,\mathrm dt^{\otimes2}}{\pm 2t^n\,\mathrm dt}=\pm\frac12\,\mathrm dt.
\]
Thus $\Psi_q$ extends holomorphically and is nonvanishing at both points lying over $p$. The case where $U$ is centered at a zero of $q$ of odd order is verified similarly. Consequently, the expression \eqref{eq:canonical-bundle-map} extends over all of $\widetilde\Sigma_q$ and gives a bundle isomorphism.
\end{proof}

\begin{remark}
In the case where $q\in\mathcal B^{\mathrm{reg}}$, Lemma~\ref{lem:canonical-bundle-isomorphism} recovers the isomorphism obtained by Hitchin \cite[Section~3.2]{hitchin2021integrable}.
\end{remark}

Write $\operatorname{div}(q)=\sum_{k=1}^{r_{\mathrm{odd}}} n_k^{\mathrm o}\cdot p_k^{\mathrm o}+\sum_{k=1}^{r_{\mathrm{even}}} n_k^{\mathrm e}\cdot p_k^{\mathrm e}$, where $n_k^{\mathrm o}$ and $n_k^{\mathrm e}$ are odd and even positive integers, respectively. Then $\widetilde{\pi}_q\colon\widetilde{\Sigma}_q\to C$ is a double cover branched precisely over the points $p_k^{\mathrm o}$. If $q$ is not the square of a holomorphic $1$-form, then $\widetilde{\Sigma}_q$ is connected, and the Riemann--Hurwitz formula together with the Lefschetz fixed-point theorem gives $g\big(\widetilde{\Sigma}_q\big)=2g-1+r_{\mathrm{odd}}/2$ and $\dim_{\mathbb C}H^{1,0}\big(\widetilde{\Sigma}_q\big)^-=g-1+r_{\mathrm{odd}}/2$. In particular, when $q\in\mathcal B^{\mathrm{reg}}$, we have $r_{\mathrm{odd}}=4g-4$, and hence $g\big(\widetilde{\Sigma}_q\big)=4g-3$ and $\dim_{\mathbb C}H^{1,0}\big(\widetilde{\Sigma}_q\big)^-=3g-3$.

\begin{corollary}\label{cor:H0-isomorphism}
Using the notation in Lemma~\ref{lem:canonical-bundle-isomorphism}, let $\widetilde D_q\coloneqq\widetilde\pi_q^*D_q\in\operatorname{Div}(\widetilde\Sigma_q)$. The map induced by $\Psi_q$, also denoted by $\Psi_q$, is a linear isomorphism
\begin{equation}\label{eq:H0-isomorphism}
H^0(C,K_C^2)\longrightarrow H^0\bigl(\widetilde\Sigma_q,K_{\widetilde\Sigma_q}(\widetilde D_q)\bigr)^-,\quad f(z)\,\mathrm d z^2\longmapsto \frac{f(\widetilde z)}{2\widetilde x}\,\mathrm d\widetilde z.
\end{equation}
\end{corollary}

\begin{proof}
The map \eqref{eq:H0-isomorphism} is clearly injective, and it remains to compare dimensions. Suppose that $q$ is not the square of a holomorphic $1$-form, and $D_q\neq0$. Since every differential in $H^0\bigl(\widetilde\Sigma_q,K_{\widetilde\Sigma_q}(\widetilde D_q)\bigr)^+$ descends to $C$,
\begin{align*}
\dim_{\mathbb C} H^0\bigl(\widetilde\Sigma_q,K_{\widetilde\Sigma_q}(\widetilde D_q)\bigr) &= 6g-6-\frac{r_{\mathrm{odd}}}{2} \quad\text{by the Riemann--Roch theorem,}\\
\dim_{\mathbb C}H^0\bigl(\widetilde\Sigma_q,K_{\widetilde\Sigma_q}(\widetilde D_q)\bigr)^+ &= \dim_{\mathbb C}H^0(C,K_C(D_q)) = 3g-3-\frac{r_{\mathrm{odd}}}{2}.
\end{align*}
Therefore both $H^0\bigl(\widetilde\Sigma_q,K_{\widetilde\Sigma_q}(\widetilde D_q)\bigr)^-$ and $H^0(C,K_C^2)$ have dimension $3g-3$. If $D_q=0$, the process is similar.

Finally, suppose that $q=\psi^2$. Then $\widetilde\Sigma_q\cong C_1\sqcup C_2$, and $\sigma_q$ exchanges the two components. Every differential in $H^0\bigl(\widetilde\Sigma_q,K_{\widetilde\Sigma_q}(\widetilde D_q)\bigr)^-$ is of the form $(\alpha,-\alpha)$ with $\alpha\in H^0(C,K_C(D_q))$. Since $D_q=\operatorname{div}(\psi)$, we have $\mathcal O_C(D_q)\cong K_C$, and hence
$H^0\bigl(\widetilde\Sigma_q,K_{\widetilde\Sigma_q}(\widetilde D_q)\bigr)^- \cong H^0(C,K_C(D_q)) \cong H^0(C,K_C^2)$.
\end{proof}

\begin{remark}
In the case where $q\in\mathcal B^{\mathrm{reg}}$, i.e. $D_q=0$, Corollary~\ref{cor:H0-isomorphism} recovers the result of Mazzeo--Swoboda--Weiss--Witt \cite[Section~2.3, Equation~(4)]{mazzeo2019asymptotic}.
\end{remark}

\subsection{Anti-invariant meromorphic differentials}\label{subsec:normalized-meromorphic-differentials}

As described in Step~3 of Subsection~\ref{subsec:intro-strategy}, the special K\"ahler metric is determined by normalized anti-invariant differentials and their integrals along the chosen $\mathfrak b$-paths. On the discriminant locus, some of these differentials become meromorphic on the normalized spectral curve, with simple poles at the preimages of the nodes. We record the required normalization results for connected and split double covers.

\begin{lemma}[Normalized differentials]\label{lem:adapted-normalized-differentials}
Let $\varpi\colon\Sigma\to C$ be a connected double cover, possibly branched, with covering involution $\sigma$. Let $\{\mathfrak a_i,\mathfrak b_i\}_{i=1}^{I}$ be a symplectic basis of $H_1(\Sigma,\mathbb R)^-$, and let $\left\{\widetilde p_k^+,\widetilde p_k^-\right\}_{k=1}^{J}$ be pairwise distinct points satisfying $\sigma(\widetilde p_k^+)=\widetilde p_k^-$. For each $1\leq k\leq J$, let $\mathfrak a_{I+k}$ be a small positively oriented loop around $\widetilde p_k^+$. Then there exists a unique basis
\[
\{\omega_i\}_{i=1}^{I+J}\subset H^0\left(\Sigma,K_\Sigma\left(\sum_{k=1}^{J}\left(\widetilde p_k^++\widetilde p_k^-\right)\right)\right)^-
\]
characterized by $\int_{\mathfrak a_i}\omega_j=\delta_{ij}$. Here $\omega_1,\ldots,\omega_I$ are holomorphic, while, for each $1\leq k\leq J$, the differential $\omega_{I+k}$ has simple poles only at $\widetilde p_k^\pm$, with $\operatorname{Res}_{\widetilde p_k^\pm}\omega_{I+k}=\pm(2\pi\mathrm i)^{-1}$.
\end{lemma}

\begin{proof}
By duality, there is a unique basis $\{\omega_i\}_{i=1}^{I}$ of $H^0(\Sigma,K_\Sigma)^-$ satisfying $\int_{\mathfrak a_i}\omega_j=\delta_{ij}$ for $1\leq i,j\leq I$. For each $1\leq k\leq J$, the integrals of these holomorphic differentials over $\mathfrak a_{I+k}$ vanish, since $\mathfrak a_{I+k}$ bounds a small disc. For such $k$, choose a meromorphic $1$-form $\eta_k$ with simple poles only at $\widetilde p_k^\pm$ and residues $\operatorname{Res}_{\widetilde p_k^\pm}\eta_k=\pm(2\pi\mathrm i)^{-1}$. Replacing $\eta_k$ by $\frac12(\eta_k-\sigma^*\eta_k)$, we may assume that it is $\sigma$-anti-invariant. The residue theorem then gives $\int_{\mathfrak a_{I+j}}\eta_k=\delta_{kj}$. Define $\omega_{I+k}\coloneqq\eta_k-\sum_{i=1}^{I}\bigl(\int_{\mathfrak a_i}\eta_k\bigr)\omega_i$. Then $\omega_{I+k}$ is anti-invariant, has the same poles and residues as $\eta_k$, and has vanishing integrals over $\mathfrak a_1,\ldots,\mathfrak a_I$. Hence the resulting collection satisfies
\[
\int_{\mathfrak a_i}\omega_j=\delta_{ij}\qquad\text{for all }1\leq i,j\leq I+J.
\]
Finally, the difference between any two collections with these properties is a holomorphic anti-invariant differential whose integrals over $\mathfrak a_1,\ldots,\mathfrak a_I$ vanish, and hence is zero by duality. This proves uniqueness.
\end{proof}

Lemma~\ref{lem:adapted-normalized-differentials} will be applied to the connected normalized spectral curves arising from the strata $\mathcal B_d$ for $1\leq d\leq2g-3$, and from the stratum $\mathcal B_{2g-2}$. For $q\in\mathcal B_{\mathrm{ab}}$, however, the normalization splits as $C_1\sqcup C_2$, and the required paths and normalized differentials must be constructed separately on each component. We abstract the corresponding normalization result as follows.

\begin{corollary}\label{cor:nested-normalized-differentials}
Let $C$ be a compact Riemann surface of genus $g$, let $\{\mathfrak a_i,\mathfrak b_i\}_{i=1}^{g}$ be a symplectic basis of $H_1(C,\mathbb R)$, and let $p_1,\ldots,p_{J+1}\in C$ be distinct points. Suppose that there are nested embedded discs
\[
\mathbb B_1\subset\mathbb B_2\subset\cdots\subset\mathbb B_J\subset C, \qquad \mathbb B_k\cap\{p_1,\ldots,p_{J+1}\}=\{p_1,\ldots,p_k\}.
\]
For $1\leq k\leq J$, let $\mathfrak a_{g+k}\coloneqq\partial\mathbb B_k$, oriented counterclockwise. Then there exists a unique basis $\{\omega_i\}_{i=1}^{g+J}$ of $H^0\bigl(C,K_C(\sum_{k=1}^{J+1}p_k)\bigr)$ characterized by $\int_{\mathfrak a_i}\omega_j=\delta_{ij}$ for $1\leq i,j\leq g+J$. Here $\omega_1,\ldots,\omega_g$ are holomorphic, while, for each $1\leq k\leq J$, the differential $\omega_{g+k}$ has simple poles only at $p_k$ and $p_{k+1}$, with residues $\operatorname{Res}_{p_k}\omega_{g+k}=(2\pi\mathrm i)^{-1}$ and $\operatorname{Res}_{p_{k+1}}\omega_{g+k}=-(2\pi\mathrm i)^{-1}$.
\end{corollary}

\begin{proof}
Let $\{\omega_i\}_{i=1}^{g}$ be the basis of holomorphic differentials dual to $\{\mathfrak a_i\}_{i=1}^{g}$. For each $1\leq k\leq J$, choose a meromorphic differential $\omega_{g+k}$ with simple poles only at $p_k$ and $p_{k+1}$ and residues $\pm(2\pi\mathrm i)^{-1}$, and subtract a suitable holomorphic differential so that its integrals over $\mathfrak a_1,\ldots,\mathfrak a_g$ vanish. The residue theorem and the nesting of the discs then give $\int_{\mathfrak a_{g+j}}\omega_{g+k}=\delta_{jk}$. The resulting differentials form the required basis.
\end{proof}

\subsection{Local asymptotic estimates}

The next two lemmas provide the local analytic estimates used in Sections~3--5. The first relates the vanishing integral to the separation of two colliding simple zeros, while the second describes the logarithmic growth of the corresponding singular integral. Together, they determine the transverse asymptotic behavior of the special K\"ahler metric near the strata $\mathcal B_d$ for $1\leq d\leq2g-3$, $\mathcal B_{2g-2}$, and $\mathcal B_{\mathrm{ab}}$.

Let $\mathbb D(r)\subset\mathbb C$ denote the disc of radius $r$, and identify a chart in $C$ with $(\mathbb D(2),z)$.

\begin{lemma}[Quadratic vanishing estimate]\label{lem:quadratic-vanishing}
Fix $q_0\in\mathcal B$ such that $q_0=z^2\,\mathrm dz^2$ on $\mathbb D(2)$. Let $\mathcal U\ni q_0$ be a sufficiently small neighborhood, and suppose that every $q\in\mathcal U$ can be written on $\mathbb D(2)$ as $q=(z-z_1(q))(z-z_2(q))g_q(z)\,\mathrm dz^2$, where $z_1(q),z_2(q)\in\mathbb D(1/2)$, and $g_q$ is holomorphic and nonvanishing on $\mathbb D(2)$, depending continuously on $q$. Then there exist constants $C_1,C_2>0$, independent of $q$, such that
\[
C_1|z_1(q)-z_2(q)|^2 \leq |\mathfrak z(q)| \leq C_2|z_1(q)-z_2(q)|^2,\quad\text{where}\quad \mathfrak z(q)\coloneqq\int_{\partial\mathbb D(1)}\sqrt q.
\]
In particular, $\mathfrak z(q)=0$ if and only if $z_1(q)=z_2(q)$, equivalently, $q$ has a double zero in $\mathbb D(1)$.
\end{lemma}

\begin{proof}
Choose the branches $h_q=\sqrt{g_q}$ continuously in $q$ so that $h_{q_0}\equiv1$, and set $\delta(q)\coloneqq z_1(q)-z_2(q)$. After the change of variables $z=t+z_2(q)$ and a slight deformation of the contour, we obtain
\[
\mathfrak z(q)=\int_{\partial\mathbb D(1)}\sqrt{t(t-\delta(q))}\,h_q(t+z_2(q))\,\mathrm dt.
\]
Write $(1-w)^{1/2}=\sum_{k=0}^{\infty}b_kw^k$, where $b_k\coloneqq(-1)^k\binom{1/2}{k}=O(k^{-3/2})$ by Stirling's formula. On $\partial\mathbb D(1)$, the uniformly convergent expansions
\[
\sqrt{t(t-\delta(q))} =\sum_{k=0}^{\infty}b_k\delta(q)^kt^{1-k}, \qquad h_q(t+z_2(q)) =\sum_{m=0}^{\infty}a_m(q)t^m
\]
hold, where the coefficients $a_m(q)$ are uniformly bounded in $m$ and $q$. Extracting the coefficient of $t^{-1}$ gives
\[
\mathfrak z(q)=2\pi\mathrm i\,\delta(q)^2S(q),
\qquad
S(q)\coloneqq\sum_{m=0}^{\infty}b_{m+2}a_m(q)\delta(q)^m.
\]
As $q\to q_0$, we have $z_1(q),z_2(q)\to0$ and $h_q\to1$, and hence $S(q)\to b_2=-1/8$. After shrinking $\mathcal U$, there exist constants $c_1,c_2>0$ such that $c_1\leq|S(q)|\leq c_2$. Consequently,
$2\pi c_1|\delta(q)|^2\leq|\mathfrak z(q)|\leq2\pi c_2|\delta(q)|^2$,
which proves the estimate. Since $S(q)\neq0$ on $\mathcal U$, we also have $\mathfrak z(q)=0$ if and only if $\delta(q)=0$, completing the proof.
\end{proof}

Here the two roots $z_1(q),z_2(q)$ cannot in general be continuously ordered for $q\in\mathcal U$. We only need to take the roots pointwise; however, $|\delta_q| = |z_1(q)-z_2(q)|$ appears in Lemma~\ref{lem:quadratic-vanishing} and makes sense on the whole $\mathcal U$.

\begin{lemma}[Local logarithmic estimate]\label{lem:local-logarithmic-estimate}
Use the notation of Lemma~\ref{lem:quadratic-vanishing}, and let $f_q$ be holomorphic on $\mathbb D(2)$, depending continuously on $q$, with $f_{q_0}(0)=(\pi\mathrm i)^{-1}$. Fix a small $r>0$. For each $q\in\mathcal U$ with $z_1(q)\neq z_2(q)$, choose $\theta_q\in\mathbb R$ such that the path $\gamma_q(t)\coloneqq z_1(q)+te^{\mathrm i\theta_q}$, $0\leq t\leq r$, does not pass through $z_2(q)$. Define
\[
\tau(q)\coloneqq-\int_{\gamma_q}\frac{f_q(z)\,\mathrm dz}{\sqrt{(z-z_1(q))(z-z_2(q))g_q(z)}},
\]
where a suitable branch of the square root is chosen. Then, after shrinking $\mathcal U$, $\operatorname{Im}\tau(q)\asymp-\log|z_1(q)-z_2(q)|$.
\end{lemma}

\begin{proof}
Choose the branches of $\sqrt{g_q}$ continuously in $q$ so that $\sqrt{g_{q_0}}\equiv1$, and define $\phi_q\coloneqq f_q/\sqrt{g_q}$. Then $\phi_q(z_2(q))\to(\pi\mathrm i)^{-1}$ as $q\to q_0$. For brevity, write $z_i=z_i(q)$ and set $\delta_q\coloneqq z_1-z_2$. Decompose $\tau(q)=\mathrm I_q+\mathrm{II}_q$,
\[
\mathrm I_q\coloneqq-\phi_q(z_2)\int_{\gamma_q}\frac{\mathrm dz}{\sqrt{(z-z_1)(z-z_2)}},\qquad
\mathrm{II}_q\coloneqq-\int_{\gamma_q}\frac{\phi_q(z)-\phi_q(z_2)}{\sqrt{(z-z_1)(z-z_2)}}\,\mathrm dz.
\]
After shrinking $\mathcal U$, there is a constant $M>0$, independent of $q$, such that $|\phi_q(z)-\phi_q(z_2)|\leq M|z-z_2|$ along $\gamma_q$. Parametrizing $\gamma_q$ by $z=z_1+te^{\mathrm i\theta_q}$ gives
\[
|\mathrm{II}_q| \leq M\int_0^r\left|\sqrt{\frac{t+\delta_qe^{-\mathrm i\theta_q}}{t}}\right|\,\mathrm dt \leq M\bigl(r+2\sqrt{r|\delta_q|}\bigr)\leq3Mr.
\]
The last inequality follows after shrinking $\mathcal U$ so that $|\delta_q|\leq r$. It remains to estimate
\[
\mathrm I_q=-\phi_q(z_2)\mathrm J_q,\quad\text{where}\quad \mathrm J_q\coloneqq\int_{\gamma_q}\frac{\mathrm dz}{\sqrt{(z-z_1)(z-z_2)}} = \int_0^r\frac{\mathrm dt}{\sqrt{t\bigl(t+\delta_qe^{-\mathrm i\theta_q}\bigr)}}.
\]
An antiderivative of the integrand is $F_q(t)\coloneqq2\log\left(\sqrt t+\sqrt{t+\delta_qe^{-\mathrm i\theta_q}}\right)$. Choose the branch of $F_q$ at $t=r$ such that $\operatorname{Im}F_q(r)\to0$ as $q\to q_0$. Since $\gamma_q$ does not pass through $z_2$, this branch can be analytically continued along $(0,r]$ and extended continuously to $t=0$. Moreover, $\operatorname{Re}F_q(0)=\log|\delta_q|$, while $\operatorname{Re}F_q(r)$ and $\operatorname{Im}F_q(r)-\operatorname{Im}F_q(0)$ remain uniformly bounded. Therefore, from $\mathrm J_q=F_q(r)-F_q(0)$, we obtain $\operatorname{Re}\mathrm J_q\asymp-\log|\delta_q|$, while $\operatorname{Im}\mathrm J_q$ remains uniformly bounded. Finally, since $-\phi_q(z_2)\to\mathrm i/\pi$ as $q\to q_0$, these estimates imply $\operatorname{Im}\mathrm I_q=\operatorname{Im}\bigl(-\phi_q(z_2)\mathrm J_q\bigr)\asymp-\log|\delta_q|$. Together with the uniform boundedness of $\mathrm{II}_q$, this proves the result.
\end{proof}

\section{The Strata $\mathcal B_d$ for $1 \leq d \leq 2g-3$}\label{sec:Bd}

Fix $q_0\in\mathcal B_d$ and let $\mathcal U$ be a sufficiently small simply connected neighborhood of $q_0$ in $\mathcal B$. Write
\[
\operatorname{div}(q_0) = 2p_{1,q_0}+2p_{3,q_0}+\cdots+2p_{2d-1,q_0} + \sum_{k=2d+1}^{4g-4}p_{k,q_0}.
\]
Following the strategy outlined in Subsection~\ref{subsec:intro-strategy}, we design adapted paths on $\widetilde\Sigma_{q_0}$, deform them to nearby spectral curves, and construct coordinates $\{\mathfrak z^i\}_{i=1}^{3g-3}$ to describe $\mathcal B_d$ locally and derive the logarithmic singular model of $\omega_{\mathrm{SK}}$.

\subsection{Adapted paths on the normalized spectral curve}\label{subsec:Bd-adapted-paths}

In this subsection, we construct a system of $3g-3$ pairs of paths $\{\mathfrak a_{i,q_0},\mathfrak b_{i,q_0}\}_{i=1}^{3g-3}$ on $\widetilde\Sigma_{q_0}$. Among them, $3g-3-d$ pairs form a symplectic basis of $H_1\big(\widetilde\Sigma_{q_0},\mathbb R\big)^-$, while the remaining $d$ pairs are associated with the preimages of the nodes. For these pairs, the $\mathfrak a$-paths are contractible loops, whereas the $\mathfrak b$-paths join the two preimages of each node. After deformation to $\Sigma_q$ for $q\in\mathcal B^{\mathrm{reg}}$, the resulting paths $\{\mathfrak a_{i,q},\mathfrak b_{i,q}\}_{i=1}^{3g-3}$ form a symplectic basis of $H_1\big(\Sigma_q,\mathbb R\big)^-$. Note that the normalized spectral curve $\widetilde\Sigma_{q_0}$ is connected, and the double cover $\widetilde\pi_{q_0}\colon\widetilde\Sigma_{q_0}\to C$ is branched precisely over the points $\{p_{k,q_0}\}_{k=2d+1}^{4g-4}$. We begin with a topological property of this cover.

\begin{definition}
A disc $\mathbb D\subset C$ is called a \emph{trivializing disc} for $\widetilde\pi_{q_0}$ if the support of the branch divisor of $\widetilde\pi_{q_0}$ is contained in $\mathbb D$ and $\widetilde\pi_{q_0}^{-1}(C\setminus\overline{\mathbb D})$ is a disjoint union of two copies of $C\setminus\overline{\mathbb D}$.
\end{definition}

Baraglia \cite[Proposition~3.1.2]{baraglia2016monodromy} proved the existence of a trivializing disc for regular spectral covers. Since the proof only uses that the cover is cyclic and branched, the same argument applies to the connected branched double cover $\widetilde\pi_{q_0}\colon\widetilde\Sigma_{q_0}\to C$. We therefore fix a trivializing disc for $\widetilde\pi_{q_0}$ and denote it by $\mathbb D\subset C$.

\begin{figure}[!t]
\centering
\includegraphics[width=0.72\linewidth]{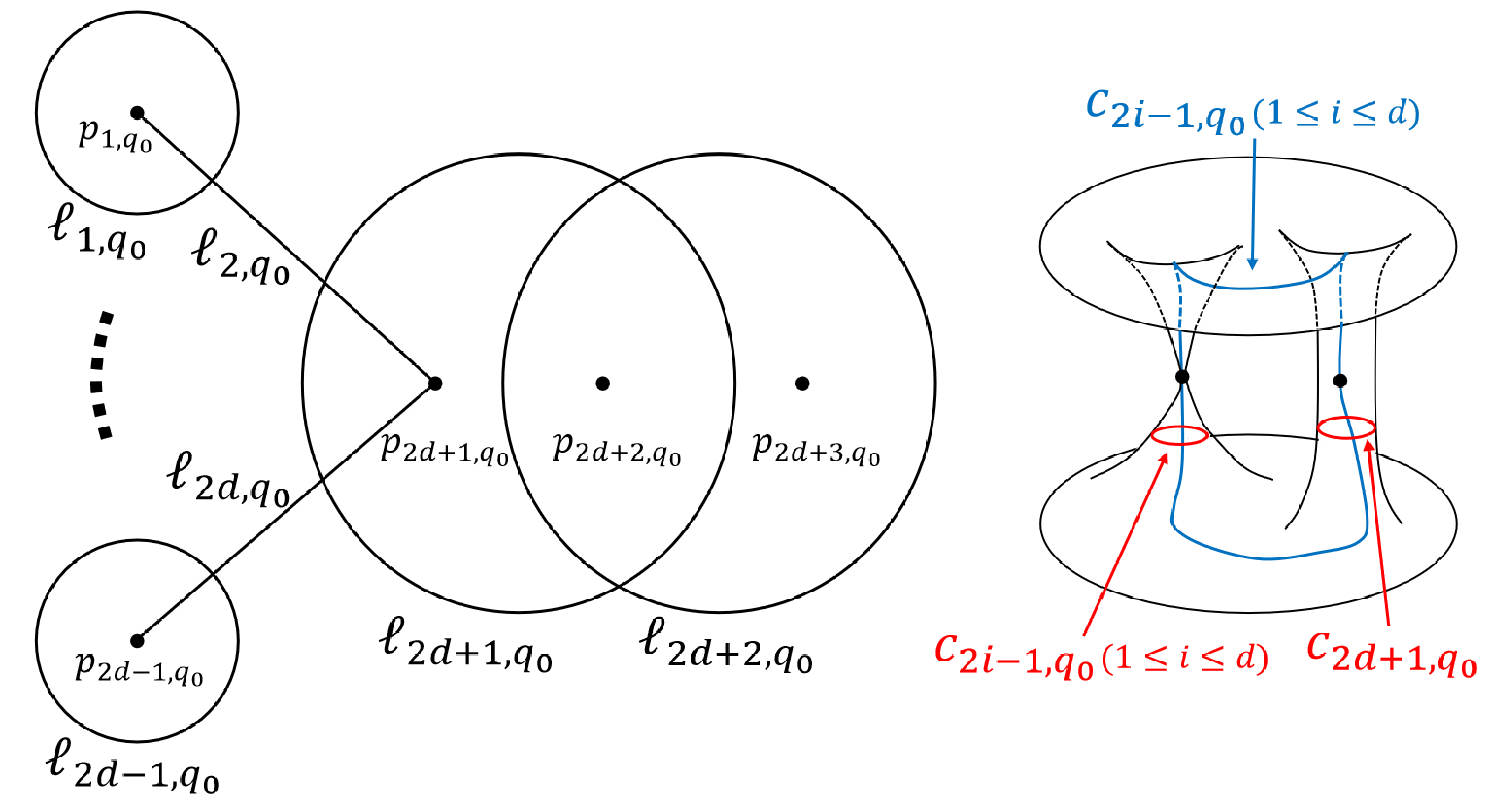}
\figalttext[Adapted paths on the base and their lifts]
{Two-panel schematic. Left: loops surround zeros and arcs connect double zeros to a fixed simple zero. Right: red loops and blue arcs show the corresponding lifts on the normalized spectral curve.}
\caption{On the left, each $\ell_{2k-1,q_0}$ for $1\leq k\leq d$ surrounds the double zero $p_{2k-1,q_0}$, while $\ell_{2k,q_0}$ joins it to the fixed simple zero $p_{2d+1,q_0}$. For $2d+1\leq m\leq4g-6$, the loop $\ell_{m,q_0}$ surrounds the consecutive simple zeros $p_{m,q_0}$ and $p_{m+1,q_0}$. On the right, the red loops are the chosen lifts $c_{2k-1,q_0}$ and $c_{m,q_0}$ of $\ell_{2k-1,q_0}$ and $\ell_{m,q_0}$, respectively, while the blue path $c_{2k,q_0}$ is the lift of $\ell_{2k,q_0}$ joining the two preimages of $p_{2k-1,q_0}$ and passing through the branch point lying over $p_{2d+1,q_0}$.}
\label{fig:Bd-adapted-paths}
\end{figure}

\begin{lemma}[Adapted path system]\label{lem:Bd-adapted-paths}
There exist $3g-3$ pairs of oriented paths $\{\mathfrak a_{i,q_0},\mathfrak b_{i,q_0}\}_{i=1}^{3g-3}$ on $\widetilde\Sigma_{q_0}$ satisfying
\[
\langle\mathfrak a_{i,q_0},\mathfrak b_{j,q_0}\rangle=\delta_{ij},\qquad
\langle\mathfrak a_{i,q_0},\mathfrak a_{j,q_0}\rangle=\langle\mathfrak b_{i,q_0},\mathfrak b_{j,q_0}\rangle=0,\qquad 1\leq i,j\leq3g-3,
\]
where $\langle\cdot,\cdot\rangle$ denotes the transverse intersection number of oriented paths. Moreover:
\begin{enumerate}
    \item The subcollection $\{\mathfrak a_{i,q_0},\mathfrak b_{i,q_0}\}_{i=1}^{g}\cup\{\mathfrak a_{i,q_0},\mathfrak b_{i,q_0}\}_{i=g+d+1}^{3g-3}$ is a symplectic basis of $H_1\big(\widetilde\Sigma_{q_0},\mathbb R\big)^-$.
    \item For each $1\leq k\leq d$, the path $\mathfrak a_{g+k,q_0}$ is a contractible loop surrounding exactly one of the two points in $\widetilde\pi_{q_0}^{-1}(p_{2k-1,q_0})$, whereas $\mathfrak b_{g+k,q_0}$ joins these two points.
\end{enumerate}
\end{lemma}

\begin{proof}
Recall that $\mathbb D$ is a trivializing disc for $\widetilde\pi_{q_0}$. Choose a symplectic basis of $H_1(C,\mathbb Z)$ represented by simple loops $\{\alpha_i,\beta_i\}_{i=1}^{g}$ in $C\setminus\overline{\mathbb D}$. Since $\widetilde\pi_{q_0}^{-1}(C\setminus\overline{\mathbb D})$ has two components, each $\alpha_i$ lifts to $a_{i,q_0}$ and $a_{i+g,q_0}$, exchanged by $\sigma_{q_0}$, and similarly $\beta_i$ lifts to $b_{i,q_0}$ and $b_{i+g,q_0}$. Define
\[
\mathfrak a_{i,q_0}\coloneqq\frac{1}{\sqrt2}\bigl(a_{i,q_0}-a_{i+g,q_0}\bigr),\qquad
\mathfrak b_{i,q_0}\coloneqq\frac{1}{\sqrt2}\bigl(b_{i,q_0}-b_{i+g,q_0}\bigr),\qquad 1\leq i\leq g.
\]
They are anti-invariant and satisfy $\langle\mathfrak a_{i,q_0},\mathfrak b_{j,q_0}\rangle=\delta_{ij}$ for $1\leq i,j\leq g$. Here we slightly abuse notation by referring to real linear combinations of loops or arcs as paths. To construct the remaining paths, we first choose auxiliary paths $\{\ell_{i,q_0}\}_{i=1}^{4g-6}$ in $\mathbb D$, whose configuration is illustrated on the left-hand side of Figure~\ref{fig:Bd-adapted-paths}.

For each $1\leq k\leq d$, choose a small disc $\mathbb B_{2k-1}\subset\mathbb D$ centered at the double zero $p_{2k-1,q_0}$ and set $\ell_{2k-1,q_0}\coloneqq\partial\mathbb B_{2k-1}$. Choose a simple arc $\ell_{2k,q_0}$ joining $p_{2k-1,q_0}$ to the fixed simple zero $p_{2d+1,q_0}$. The arcs $\ell_{2,q_0},\ell_{4,q_0},\ldots,\ell_{2d,q_0}$ may be chosen to meet only at $p_{2d+1,q_0}$, with distinct tangent directions there, and to be otherwise pairwise disjoint. Moreover, each $\ell_{2k,q_0}$ intersects transversely only $\ell_{2k-1,q_0}$ and $\ell_{2d+1,q_0}$.

For $2d+1\leq k\leq4g-6$, choose a small disc $\mathbb B_k\subset\mathbb D$ containing precisely the two simple zeros $p_{k,q_0}$ and $p_{k+1,q_0}$, and set $\ell_{k,q_0}\coloneqq\partial\mathbb B_k$. The discs are chosen so that $\ell_{k,q_0}$ and $\ell_{k',q_0}$ are disjoint unless $|k-k'|=1$, in which case they intersect transversely in two points. We may also arrange that each $\ell_{2i,q_0}$, $1\leq i\leq d$, is disjoint from all these loops except $\ell_{2d+1,q_0}$.

Next, we lift these auxiliary paths to $\widetilde\Sigma_{q_0}$, whose configuration is illustrated on the right-hand side of Figure~\ref{fig:Bd-adapted-paths}. For each $1\leq k\leq d$, the inverse image $\widetilde\pi_{q_0}^{-1}(\ell_{2k-1,q_0})$ consists of two contractible loops; choose one and denote it by $c_{2k-1,q_0}$. Define $c_{2k,q_0}\coloneqq\widetilde\pi_{q_0}^{-1}(\ell_{2k,q_0})$, which is a path joining the two preimages of $p_{2k-1,q_0}$ and passing through $\widetilde\pi_{q_0}^{-1}(p_{2d+1,q_0})$. For $2d+1\leq k\leq4g-6$, the inverse image $\widetilde\pi_{q_0}^{-1}(\ell_{k,q_0})$ also consists of two noncontractible loops; choose one and denote it by $c_{k,q_0}$. The orientations of $\{c_{k,q_0}\}_{k=1}^{4g-6}$ may be chosen so that
\begin{align*}
&\left\langle c_{2k-1,q_0},c_{2k,q_0}\right\rangle = \left\langle c_{2k,q_0},c_{2d+1,q_0}\right\rangle=1,\qquad 1\leq k\leq d,\\
&\left\langle c_{k,q_0},c_{k+1,q_0}\right\rangle=1,\qquad 2d+1\leq k\leq4g-7.
\end{align*}
For $1\leq i<j\leq d$, the paths $c_{2i,q_0}$ and $c_{2j,q_0}$ meet transversely at the unique point of $\widetilde\pi_{q_0}^{-1}(p_{2d+1,q_0})$, since the arcs $\ell_{2i,q_0}$ and $\ell_{2j,q_0}$ meet at $p_{2d+1,q_0}$ with distinct tangent directions. Set $n_{ij}\coloneqq\langle c_{2i,q_0},c_{2j,q_0}\rangle\in\{\pm1\}$.
Finally, define the paths $\{\mathfrak a_{i,q_0},\mathfrak b_{i,q_0}\}_{i=g+1}^{3g-3}$ on $\widetilde\Sigma_{q_0}$ by
\begin{equation}
\label{eq:Bd-adapted-paths-definition}
\begin{aligned}
\mathfrak a_{g+i,q_0}&\coloneqq c_{2i-1,q_0}, & \mathfrak b_{g+i,q_0}&\coloneqq c_{2i,q_0}-\sum_{k=i+1}^{d}n_{ik}c_{2k-1,q_0},
&1\leq i\leq d,\\[-3mm]
\mathfrak a_{g+k,q_0}&\coloneqq\sum_{j=1}^{k}c_{2j-1,q_0}, & \mathfrak b_{g+k,q_0}&\coloneqq c_{2k,q_0}, &d+1\leq k\leq2g-3.
\end{aligned}
\end{equation}
Note that $\{\mathfrak a_{i,q_0},\mathfrak b_{i,q_0}\}_{i=1}^{g}$ are disjoint from $\{\mathfrak a_{i,q_0},\mathfrak b_{i,q_0}\}_{i=g+1}^{3g-3}$. By a direct computation, we have
\[
\langle\mathfrak a_{i,q_0},\mathfrak b_{j,q_0}\rangle=\delta_{ij},\qquad\langle\mathfrak a_{i,q_0},\mathfrak a_{j,q_0}\rangle= \langle\mathfrak b_{i,q_0},\mathfrak b_{j,q_0}\rangle=0,\qquad 1\leq i,j\leq3g-3.
\]
Finally, $\dim_{\mathbb R}H_1(\widetilde\Sigma_{q_0},\mathbb R)^-=2(3g-3-d)$, and the collection $\{\mathfrak a_{i,q_0},\mathfrak b_{i,q_0}\}_{i=1}^{g} \cup \{\mathfrak a_{i,q_0},\mathfrak b_{i,q_0}\}_{i=g+d+1}^{3g-3}$ consists of $3g-3-d$ symplectic pairs representing anti-invariant classes. It forms a symplectic basis of $H_1\big(\widetilde\Sigma_{q_0},\mathbb R\big)^-$.
\end{proof}

\subsection{Deformation of paths and conjugate coordinates}\label{subsec:Bd-deformation}

Recall that $\mathcal U\subset\mathcal B$ is a sufficiently small simply connected neighborhood of $q_0$. In this subsection, we deform the adapted paths $\{\mathfrak a_{i,q_0},\mathfrak b_{i,q_0}\}_{i=1}^{3g-3}$ to $\widetilde\Sigma_q$ for each $q\in\mathcal U$. The deformations of the paths arising from auxiliary arcs connecting zeros of $q_0$ are not uniquely determined when a double zero of $q_0$ splits into two simple zeros of $q$. By a continuous deformation of a path, we mean a family $\{\gamma_q\}_{q\in\mathcal U}$, with $\gamma_q\subset\widetilde\Sigma_q$, such that the images $\nu_q(\gamma_q)$ vary continuously in $\bigsqcup_{q\in\mathcal U}\Sigma_q\subset \operatorname{Tot}(K_C)\times\mathcal U$. 

The discs $\{\mathbb B_{2k-1}\}_{k=1}^d$ are centered at the double zeros $p_{2k-1,q_0}$, while $\{\mathbb B_k\}_{k=2d+1}^{4g-6}$ contain the pairs of simple zeros $p_{k,q_0}$ and $p_{k+1,q_0}$. After shrinking $\mathcal U$ if necessary, we may assume the following for every $q\in\mathcal U$:
\begin{itemize}
    \item each disc $\mathbb B_{2k-1}$, $1\leq k\leq d$, contains either one double zero or two simple zeros of $q$;
    \item each disc $\mathbb B_k$, $2d+1\leq k\leq4g-6$, contains exactly two simple zeros of $q$;
    \item there exists a fixed disc $V\subset\mathbb D$ centered at $p_{2d+1,q_0}$ and disjoint from all other zeros of $q_0$, such that $V$ contains exactly one simple zero of $q$, denoted by $p_{2d+1,q}$.
\end{itemize}

Let $\{\alpha_i,\beta_i\}_{i=1}^g$ be the symplectic basis of $H_1(C,\mathbb Z)$ chosen in Subsection~\ref{subsec:Bd-adapted-paths}. For each $q\in\mathcal U$, let $a_{i,q}$ and $a_{i+g,q}$ be the two lifts of $\alpha_i$ via $\widetilde\pi_q$, and define $b_{i,q}$ and $b_{i+g,q}$ similarly. The deformation of $\{\mathfrak a_{i,q_0},\mathfrak b_{i,q_0}\}_{i=1}^{g}$ is given by
\[
\mathfrak a_{i,q}\coloneqq\frac{1}{\sqrt2}\bigl(a_{i,q}-a_{i+g,q}\bigr),\qquad \mathfrak b_{i,q}\coloneqq\frac{1}{\sqrt2}\bigl(b_{i,q}-b_{i+g,q}\bigr),\qquad 1\leq i\leq g.
\]
Since their projections to $C$ are fixed, these paths vary continuously with $q$. The same construction applies to $\{\mathfrak a_{i,q_0}\}_{i=g+1}^{3g-3}\cup\{\mathfrak b_{i,q_0}\}_{i=g+d+1}^{3g-3}$. More precisely, for $k\in\{2i-1\mid 1\leq i\leq d\}\cup\{2d+1,\ldots,4g-6\}$, keep the auxiliary paths fixed by setting $\ell_{k,q}\coloneqq\ell_{k,q_0}$. After shrinking $\mathcal U$, each $\ell_{k,q}$ avoids the zeros of $q$, and $\widetilde\pi_q^{-1}(\ell_{k,q})$ has two components. Choose the component $c_{k,q}$ continuously in $q$ so that it agrees with $c_{k,q_0}$ at $q=q_0$. The paths $\{\mathfrak a_{i,q}\}_{i=g+1}^{3g-3}\cup\{\mathfrak b_{i,q}\}_{i=g+d+1}^{3g-3}$ are then defined by the same formulas as in \eqref{eq:Bd-adapted-paths-definition}.

\begin{lemma}\label{lem:Bd-symplectic-basis}
$\{\mathfrak a_{i,q},\mathfrak b_{i,q}\}_{i=1}^{g} \cup \{\mathfrak a_{i,q},\mathfrak b_{i,q}\}_{i=g+d+1}^{3g-3}$ forms a symplectic basis of $H_1\big(\widetilde\Sigma_q,\mathbb R\big)^-$ for $q\in\mathcal U\cap\mathcal B_d$.
\end{lemma}

\begin{proof}
The paths remain anti-invariant, and their transverse intersection numbers are preserved under continuous deformation. The result therefore follows from Lemma~\ref{lem:Bd-adapted-paths}.
\end{proof}

Before deforming $\{\mathfrak b_{i,q_0}\}_{i=g+1}^{g+d}$, we distinguish two local models for the deformation of an auxiliary path. Figure~\ref{fig:deformation-type} illustrates the corresponding local configurations near an endpoint.

\begin{figure}[!t]
\centering
\includegraphics[width=0.84\linewidth]{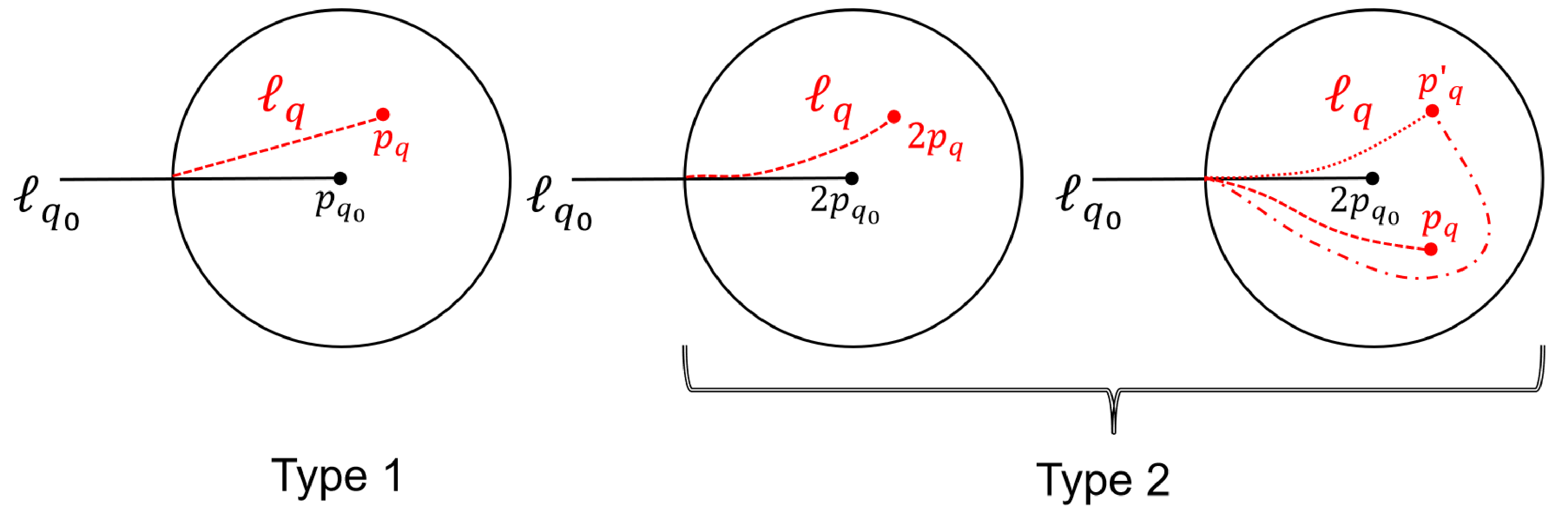}
\figalttext[Local deformation models for auxiliary arcs]
{Three-panel schematic. Black curves show the original arcs and red curves their deformations. From left to right, the endpoint is a moving simple zero, a persistent double zero, or one of two simple zeros created by splitting.}
\caption{Local models for the deformation of an auxiliary arc $\ell_{q_0}$ near one of its endpoints. In each panel, the black curve represents the original arc $\ell_{q_0}$, and the red curve represents a deformed arc on a nearby spectral curve. Left: Type~1 deformation, where the endpoint of $\ell_{q_0}$ is a simple zero $p_{q_0}$; after deformation, the endpoint moves to the unique nearby simple zero $p_q$. Middle and right: Type~2 deformation, where the endpoint of $\ell_{q_0}$ is a double zero $p_{q_0}$; after deformation, the endpoint is replaced either by a nearby double zero or by one of the two simple zeros arising from the splitting of $p_{q_0}$.}
\label{fig:deformation-type}
\end{figure}

\begin{itemize}
    \item[Type 1] Suppose first that one endpoint of an auxiliary path $\ell_{q_0}$ is a simple zero $p_{q_0}$ of $q_0$. Choose a small disc $\mathbb B$ centered at $p_{q_0}$ and fix a coordinate chart on $\mathbb B$. After shrinking $\mathcal U$, every $q\in\mathcal U$ has a unique simple zero $p_q$ in $\mathbb B$. We define $\ell_q$ to coincide with $\ell_{q_0}$ outside $\mathbb B$, while inside $\mathbb B$ it is the straight-line segment joining $\ell_{q_0}\cap\partial\mathbb B$ to $p_q$ in the fixed coordinate chart. Once the disc and the coordinate chart are fixed, this deformation is determined.

    \item[Type 2] Suppose instead that one endpoint of $\ell_{q_0}$ is a double zero $p_{q_0}$ of $q_0$. Choose a small disc $\mathbb B$ centered at $p_{q_0}$. For $q\in\mathcal U$, the zeros of $q$ in $\mathbb B$ consist either of one double zero or of two simple zeros. We define a deformation of $\ell_{q_0}$ to coincide with $\ell_{q_0}$ outside $\mathbb B$. Inside $\mathbb B$, it is any simple arc joining $\ell_{q_0}\cap\partial\mathbb B$ to a zero of $q$ in $\mathbb B$. When the double zero splits into two simple zeros, the arc is required to avoid the zero that is not chosen as its endpoint.
\end{itemize}

In Type~2, the deformation is not uniquely determined. We denote by $(\ell_q)$ the collection of all curves obtained in this way, and refer to each member of $(\ell_q)$ as a deformation of $\ell_{q_0}$. We now apply the two deformation types above to the auxiliary paths $\{\ell_{2k,q_0}\}_{k=1}^{d}$ constructed in Subsection~\ref{subsec:Bd-adapted-paths}. Figure~\ref{fig:auxiliary-curve-deformation} illustrates such a deformation for $\ell_{2,q_0}$.

Note that each $\ell_{2k,q_0}$ has endpoints at $p_{2k-1,q_0}$ and $p_{2d+1,q_0}$: the point $p_{2k-1,q_0}$ is a double zero of $q_0$, with neighborhood $\mathbb B_{2k-1}$, whereas $p_{2d+1,q_0}$ is a simple zero of $q_0$, with neighborhood $V$. Accordingly, we keep $\ell_{2k,q_0}$ fixed outside $\mathbb B_{2k-1}\cup V$, deform it by Type~2 near $p_{2k-1,q_0}$, and by Type~1 near $p_{2d+1,q_0}$. Denote by $(\ell_{2k,q})$ the collection of all deformations of $\ell_{2k,q_0}$ obtained in this way.

\begin{figure}[t!]
\centering
\includegraphics[width=0.6\linewidth]{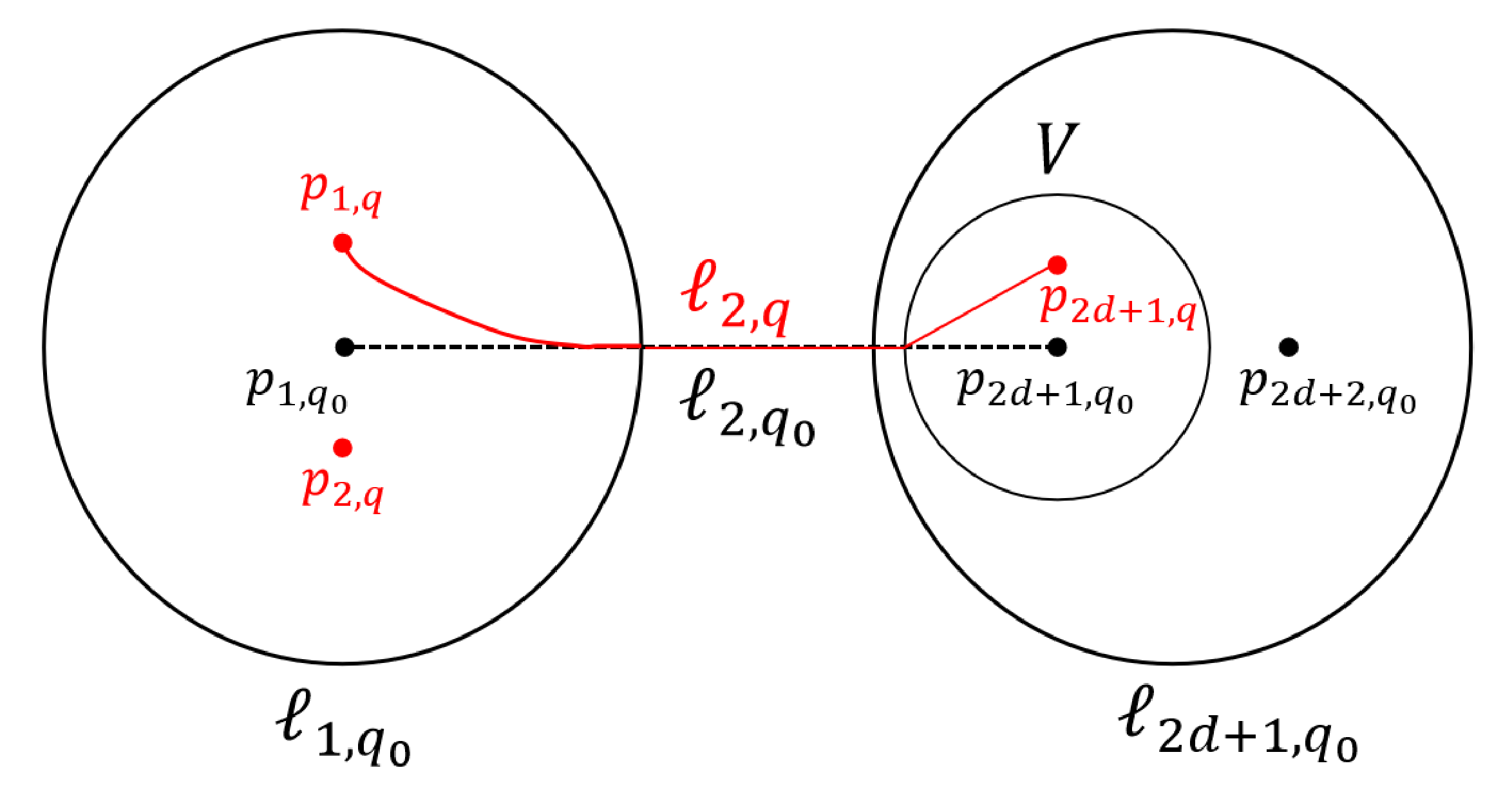}
\figalttext[Deformation of an auxiliary path]
{Schematic deformation of an auxiliary path joining a double zero to a simple zero. The original path is black and the deformed path red, with Type~2 deformation at the double-zero end and Type~1 deformation at the simple-zero end.}
\caption{The red curve is a schematic illustration of a deformation of the auxiliary path $\ell_{2k,q_0}$. The endpoint near the double zero is deformed according to Type~2, while the endpoint near the simple zero is deformed according to Type~1.}
\label{fig:auxiliary-curve-deformation}
\end{figure}

Lifting the paths in $(\ell_{2k,q})$ via $\widetilde\pi_q$ gives a corresponding collection of paths $(c_{2k,q})$ on $\widetilde\Sigma_q$. Despite the freedom in choosing these deformations, their lifts have a simple homological ambiguity. If $q$ has two simple zeros in $\mathbb B_{2k-1}$, then any two choices $c_{2k,q}$ and $c'_{2k,q}$ satisfy
\[
[c_{2k,q} - c'_{2k,q}]=n_k[c_{2k-1,q}],\qquad n_k\in\mathbb Z.
\]
The integer $n_k$ depends on the choice of the endpoint of $\ell_{2k,q}$ in $\mathbb B_{2k-1}$ and on the number of times the path winds around the other zero. For any choice of the paths $\{c_{2k,q}\}_{k=1}^{d}$, define $\{\mathfrak a_{i,q},\mathfrak b_{i,q}\}_{i=g+1}^{3g-3}$ from $\{c_{k,q}\}_{k=1}^{4g-6}$ by the same formulas as in \eqref{eq:Bd-adapted-paths-definition}. In particular, for each $1\leq k\leq d$, the path $\mathfrak b_{g+k,q}$ is not uniquely determined. We denote the collection of all such paths by $(\mathfrak b_{g+k,q})$. If $q\in\mathcal U\cap\mathcal B^{\mathrm{reg}}$ and $\mathfrak b_{g+k,q},\mathfrak b'_{g+k,q}\in(\mathfrak b_{g+k,q})$, then
\begin{equation}\label{eq:Bd-b-ambiguity}
\left[\mathfrak b_{g+k,q} - \mathfrak b'_{g+k,q}\right]=n_k\left[\mathfrak a_{g+k,q}\right],\qquad n_k\in\mathbb Z.
\end{equation}
Although there is no canonical global choice of these paths, continuous choices exist locally near every regular point, as shown in the following lemma.

\begin{lemma}[Local continuous choice]\label{lem:Bd-deformation}
Fix $q\in\mathcal U\cap\mathcal B^{\mathrm{reg}}$. There exists a sufficiently small neighborhood $\mathcal V\subset\mathcal U\cap\mathcal B^{\mathrm{reg}}$ of $q$ such that, for each $1\leq i\leq d$, one can choose $c_{2i,q'}\in(c_{2i,q'})$ for every $q'\in\mathcal V$ so that the family $\{c_{2i,q'}\}_{q'\in\mathcal V}$ varies continuously with $q'$.

Consequently, the paths $\{\mathfrak b_{g+i,q'}\}_{i=1}^{d}$ defined by \eqref{eq:Bd-adapted-paths-definition} vary continuously with $q'$. Moreover, for every $q'\in\mathcal V$, the collection $\{\mathfrak a_{i,q'},\mathfrak b_{i,q'}\}_{i=1}^{3g-3}$ forms a symplectic basis of $H_1(\Sigma_{q'},\mathbb R)^-$.
\end{lemma}

\begin{proof}
For each $1\leq k\leq d$, the disc $\mathbb B_{2k-1}$ contains two distinct simple zeros of $q$, since $q\in\mathcal B^{\mathrm{reg}}$. Choose one of them and let $W_k\subset\mathbb B_{2k-1}$ be a sufficiently small disc centered at the chosen zero and disjoint from the other one. After shrinking $\mathcal V$ if necessary, we may assume that $W_k$ and the fixed disc $V$ defined above each contain exactly one simple zero of every $q'\in\mathcal V$.

Choose an arc $\ell_{2k,q}\in(\ell_{2k,q})$ whose portions inside $W_k$ and $V$ are straight-line segments in the fixed coordinate charts. For each $q'\in\mathcal V$, define $\ell_{2k,q'}$ to coincide with $\ell_{2k,q}$ outside $W_k\cup V$. Inside $W_k$ and $V$, deform $\ell_{2k,q}$ according to Type~1, since each disc contains a unique simple zero of $q'$. A Type~1 deformation is uniquely determined once the discs and coordinate charts are fixed, and hence this construction gives a continuous family $\{\ell_{2k,q'}\}_{q'\in\mathcal V}$. Their inverse images under $\widetilde\pi_{q'}$ form a continuous family $\{c_{2k,q'}\}_{q'\in\mathcal V}$. It follows from \eqref{eq:Bd-adapted-paths-definition} that the paths $\mathfrak b_{g+k,q'}$, $1\leq k\leq d$, also vary continuously with $q'$. Intersection numbers are preserved under continuous deformation. Therefore, $\langle\mathfrak a_{i,q'},\mathfrak b_{j,q'}\rangle=\delta_{ij}$ and $\langle\mathfrak a_{i,q'},\mathfrak a_{j,q'}\rangle=\langle\mathfrak b_{i,q'},\mathfrak b_{j,q'}\rangle=0$. These paths constitute a symplectic basis of $H_1(\Sigma_{q'},\mathbb R)^-$. 
\end{proof}

\begin{remark}
The key point in Lemma~\ref{lem:Bd-deformation} is that, near a fixed $q\in\mathcal B^{\mathrm{reg}}$, the two simple zeros contained in $\mathbb B_{2i-1}$ are locally distinguishable. More precisely, after choosing one of them, the disc $W_i\subset\mathbb B_{2i-1}$ can be chosen so that it contains exactly one zero of every $q'\in\mathcal V$, while the other zero remains outside $W_i$. By contrast, as the two zeros of $q$ coalesce, such a fixed separating disc can no longer be maintained.
\end{remark}

We now define the following integral functions:
\begin{equation}\label{eq:Bd-period-coordinates}
\mathfrak z^i(q)\coloneqq\int_{\mathfrak a_{i,q}}\vartheta_q,\qquad \mathfrak w_i(q)\coloneqq-\int_{\mathfrak b_{i,q}}\vartheta_q,\qquad 1\leq i\leq3g-3.
\end{equation}
The functions $\{\mathfrak z^k\}_{k=1}^{3g-3}\cup\{\mathfrak w_k\}_{k=1}^{g}\cup\{\mathfrak w_k\}_{k=g+d+1}^{3g-3}$ are single-valued and holomorphic on $\mathcal U$.

The function $\mathfrak w_{g+k}$ is naturally multivalued on $\{q\in\mathcal U\mid q\text{ has two simple zeros in }\mathbb B_{2k-1}\}$, since the path $\mathfrak b_{g+k,q}$ is not uniquely determined. By \eqref{eq:Bd-b-ambiguity}, any two local branches satisfy
\begin{equation}\label{eq:Bd-w-multivalued}
\mathfrak w_{g+k}(q)=\mathfrak w'_{g+k}(q)-n_k\mathfrak z^{g+k}(q),\qquad n_k\in\mathbb Z.
\end{equation}
Branches of $\mathfrak w_{g+k}$ admit continuous extensions to $\mathcal U$, which agree along $\{q\in\mathcal U\mid q\text{ has a double zero in }\mathbb B_{2k-1}\}$.

Holomorphic branches of $\{\mathfrak w_{g+k}\}_{k=1}^{d}$ can be chosen on every sufficiently small simply connected open subset of $\mathcal U\cap\mathcal B^{\mathrm{reg}}$; together with the remaining functions, they form a system of conjugate special coordinates $\{\mathfrak z^i,\mathfrak w_i\}_{i=1}^{3g-3}$ for the special K\"ahler structure on $\mathcal B^{\mathrm{reg}}$.

\subsection{Singular model of the special K\"ahler metric near $\mathcal B_d$}\label{subsec:Bd-singular-model}

We now use the conjugate special coordinates constructed in Subsection~\ref{subsec:Bd-deformation} to describe the behavior of $\omega_{\mathrm{SK}}$ near $\mathcal B_d$. We first characterize $\mathcal B_d$ locally as the common zero locus of $\{\mathfrak z^{g+k}\}_{k=1}^{d}$ and determine its tangent space. We then express the special K\"ahler metric on $\mathcal B^{\mathrm{reg}}$ in these coordinates and use normalized anti-invariant differentials to derive the logarithmic asymptotics of its coefficients.

\begin{lemma}\label{lem:Bd-vanishing-differential}
For $1\leq k\leq d$ and $\dot q\in T_{q_0}\mathcal B\cong H^0(C,K_C^2)$, $\mathrm d\mathfrak z^{g+k}|_{q_0}(\dot q)=0$ if and only if $\dot q$ vanishes at $p_{2k-1,q_0}$.
\end{lemma}

\begin{proof}
Recall that $\mathbb B_{2k-1}$ is centered at $p_{2k-1,q_0}$ and that $\mathfrak a_{g+k,q}$ is a lift of $\partial\mathbb B_{2k-1}$ via $\widetilde\pi_q$. Choose a coordinate $z$ on $\mathbb B_{2k-1}$ such that $q_0=z^2\,\mathrm dz^2$ and $\dot q=f(z)\,\mathrm dz^2$. We have $\displaystyle\sqrt{z^2+t f(z)}=z+\frac{t}{2}\frac{f(z)}{z}+O(t^2)$ uniformly along $\partial\mathbb B_{2k-1}$ after choosing the branch of the square root that agrees with $z$ at $t=0$. Hence, by \eqref{eq:Bd-period-coordinates},
\[
\mathfrak z^{g+k}(q_0+t\dot q)=\int_{\partial\mathbb B_{2k-1}}\sqrt{z^2+t f(z)}\,\mathrm dz=\frac{t}{2}\int_{\partial\mathbb B_{2k-1}}\frac{f(z)}{z}\,\mathrm dz+O(t^2)=\pi\mathrm i\,f(0)t+O(t^2) \quad \text{as } t \to 0.
\]
Therefore $\mathrm d\mathfrak z^{g+k}|_{q_0}(\dot q)=\pi\mathrm i\,f(0)$, which vanishes precisely when $\dot q$ vanishes at $p_{2k-1,q_0}$.
\end{proof}

The preceding lemma shows that the differentials $\{\mathrm d\mathfrak z^{g+k}\}_{k=1}^{d}$ detect the first-order deformations that preserve the corresponding double zeros. We now identify $\mathcal B_d$ locally as the common zero locus of these coordinates.

\begin{theorem}\label{thm:Bd-zero-locus}
After shrinking $\mathcal U$ if necessary,
\[
\mathcal B_d\cap\mathcal U=\bigl\{q\in\mathcal U\mid \mathfrak z^{g+1}(q)=\cdots=\mathfrak z^{g+d}(q)=0\bigr\}.
\]
In particular, $\mathcal B_d\cap\mathcal U$ is a complex submanifold of $\mathcal U$ of dimension $3g-3-d$.
\end{theorem}

\begin{proof}
Note that $|\mathfrak z^{g+k}(q)|=\left|\int_{\partial\mathbb B_{2k-1}}\sqrt q\right|$ for $1\leq k\leq d$ by \eqref{eq:Bd-period-coordinates}. By Lemma~\ref{lem:quadratic-vanishing}, this integral vanishes if and only if $q$ has a double zero in $\mathbb B_{2k-1}$. After shrinking $\mathcal U$ as above, this proves the asserted description of $\mathcal B_d\cap\mathcal U$. It remains to show that $\mathrm d\mathfrak z^{g+1},\ldots,\mathrm d\mathfrak z^{g+d}$ are linearly independent at $q_0$. The Riemann--Roch theorem gives
\[
\dim H^0\left(C,K_C^2\left(-\sum_{j\neq k}p_{2j-1,q_0}\right)\right)=\dim H^0\left(C,K_C^2\left(-\sum_{j=1}^{d}p_{2j-1,q_0}\right)\right)+1,\qquad 1\leq k\leq d.
\]
Hence there exists $\dot q_k\in H^0(C,K_C^2)$ that vanishes at $p_{2j-1,q_0}$ for every $j\neq k$ but not at $p_{2k-1,q_0}$. After rescaling $\dot q_k$, Lemma~\ref{lem:Bd-vanishing-differential} yields $\mathrm d\mathfrak z^{g+j}|_{q_0}(\dot q_k)=\delta_{jk}$. Thus these differentials are linearly independent at $q_0$, and the implicit function theorem shows that $\mathcal B_d\cap\mathcal U$ is a complex submanifold of codimension $d$.
\end{proof}

\begin{corollary}\label{cor:Bd-tangent-space}
Under the identification $T_{q_0}\mathcal B\cong H^0(C,K_C^2)$, the tangent space of $\mathcal B_d$ at $q_0$ is $T_{q_0}\mathcal B_d = \bigcap_{k=1}^{d} \ker\bigl(\mathrm d\mathfrak z^{g+k}|_{q_0}\bigr) = H^0\bigl(C,K_C^2(-D_{q_0})\bigr)$.
\end{corollary}

\begin{remark}
He--Horn--Li \cite[Lemma~2.2]{he2025asymptotics} also studied the smoothness and tangent space of $\mathcal B_d$. Theorem~\ref{thm:Bd-zero-locus} further identifies $\mathcal B_d$ locally as the common zero locus of $\mathfrak z^{g+1},\ldots,\mathfrak z^{g+d}$.
\end{remark}

For convenience, set
\[
B\coloneqq\{1,\ldots,g\}\cup\{g+d+1,\ldots,3g-3\}, \qquad B^c\coloneqq\{g+1,\ldots,g+d\}.
\]
On a sufficiently small simply connected open subset of $\mathcal U\cap\mathcal B^{\mathrm{reg}}$, choose a single-valued branch of each $\mathfrak w_j$, $j\in B^c$. Then $\omega_{\mathrm{SK}}$ has the local expression
\begin{equation}\label{eq:Bd-SK-metric}
\omega_{\mathrm{SK}}=\frac{\mathrm i}{2}\sum_{i,j=1}^{3g-3}\operatorname{Im}(\tau_{ij})\,\mathrm d\mathfrak z^i\wedge\mathrm d\overline{\mathfrak z}^j,\qquad \tau_{ij}\coloneqq-\frac{\partial\mathfrak w_j}{\partial\mathfrak z^i}.
\end{equation}
By \eqref{eq:Bd-w-multivalued}, $\tau_{ij}'=\tau_{ij}-n_j\delta_{ij}$, and hence $\operatorname{Im}(\tau_{ij}')=\operatorname{Im}(\tau_{ij})$. Thus \eqref{eq:Bd-SK-metric} is independent of the choice of branches. We have another expression for the coefficients $\operatorname{Im}(\tau_{ij})$ in terms of normalized meromorphic differentials as follows.

\begin{lemma}\label{lem:Bd-dual-forms}
For every $q\in\mathcal U$, there exists a unique basis $\{\omega_{i,q}\}_{i=1}^{3g-3}$ of $H^0\bigl(\widetilde\Sigma_q,K_{\widetilde\Sigma_q}(\widetilde D_q)\bigr)^-$ such that
\[
\int_{\mathfrak a_{i,q}}\omega_{j,q}=\delta_{ij},\qquad 1\leq i,j\leq3g-3.
\]
At $q=q_0$, the form $\omega_{i,q_0}$ is holomorphic for every $i\in B$, while, for each $1\leq k\leq d$, the form $\omega_{g+k,q_0}$ has simple poles only at the two points of $\widetilde\pi_{q_0}^{-1}(p_{2k-1,q_0})$, with residues $\pm(2\pi\mathrm i)^{-1}$.
\end{lemma}

\begin{proof}
For $q\in\mathcal U$, let $I_q$ be the set of indices $k\in\{1,\ldots,d\}$ for which $q$ has a double zero in $\mathbb B_{2k-1}$. For each $k\in I_q$, the path $\mathfrak a_{g+k,q}$ is contractible, whereas the remaining pairs
\[
\{\mathfrak a_{i,q},\mathfrak b_{i,q}\}_{i=1}^{g} \cup \{\mathfrak a_{g+k,q},\mathfrak b_{g+k,q}\}_{k\in\{1,\ldots,d\}\setminus I_q} \cup \{\mathfrak a_{i,q},\mathfrak b_{i,q}\}_{i=g+d+1}^{3g-3}
\]
form a symplectic basis of $H_1(\widetilde\Sigma_q,\mathbb R)^-$. All the assertions follow from Lemma~\ref{lem:adapted-normalized-differentials}.
\end{proof}

To compare the normalized differentials on the varying curves $\widetilde\Sigma_q$, use the isomorphism $\Psi_q$ in \eqref{eq:H0-isomorphism} and set $Q_{i,q}\coloneqq\Psi_q^{-1}(\omega_{i,q})\in H^0(C,K_C^2)$. Accordingly, continuity of $\omega_{i,q}$ with respect to $q$ will mean continuity of $Q_{i,q}$ in the fixed vector space $H^0(C,K_C^2)$.

\begin{lemma}[Parameter dependence of normalized differentials]
\label{lem:family-normalization}
For each $1\leq i\leq3g-3$, the differential $Q_{i,q}$ depends continuously on $q\in\mathcal U$.
\end{lemma}

\begin{proof}
The collection $\{Q_{i,q_0}\}_{i=1}^{3g-3}$ is a basis of $H^0(C,K_C^2)$. Define
\[
M_{jk}(q)\coloneqq
\int_{\mathfrak a_{j,q}}\Psi_q(Q_{k,q_0}) = \int_{\widetilde\pi_q(\mathfrak a_{j,q})} \frac{Q_{k,q_0}}{2\sqrt q},\qquad 1\leq j,k\leq3g-3,
\]
where the branch of $\sqrt q$ along the projected path is determined by $\mathfrak a_{j,q}$. By the construction of the adapted paths, the projected integration paths are independent of $q$ and remain disjoint from the zeros of every $q\in\mathcal U$. Hence $M(q)$ depends continuously on $q$. Moreover, $M_{jk}(q_0)=\int_{\mathfrak a_{j,q_0}}\omega_{k,q_0}=\delta_{jk}$,
so $M(q_0)=I$. After shrinking $\mathcal U$, the matrix $M(q)$ is therefore invertible for every $q\in\mathcal U$. Set
\[
\widehat Q_{i,q}\coloneqq\sum_{k=1}^{3g-3}\bigl(M(q)^{-1}\bigr)_{ki}Q_{k,q_0}.
\]
Then $\int_{\mathfrak a_{j,q}}\Psi_q(\widehat Q_{i,q}) = \delta_{ji}$. By the uniqueness in Lemma~\ref{lem:Bd-dual-forms},
$\Psi_q(\widehat Q_{i,q})=\omega_{i,q}$, and hence $\widehat Q_{i,q}=Q_{i,q}$. The continuity of $Q_{i,q}$ now follows from that of $M(q)^{-1}$.
\end{proof}

We can now state the singular model of $\omega_{\mathrm{SK}}$ near $\mathcal B_d$.

\begin{theorem}\label{thm:Bd-singular-model}
For $q\in\mathcal U\cap\mathcal B^{\mathrm{reg}}$, the special K\"ahler metric and the matrix $(\tau_{ij})$ are given by
\[
\omega_{\mathrm{SK}}=\frac{\mathrm i}{2}\sum_{i,j=1}^{3g-3}\operatorname{Im}(\tau_{ij})\,\mathrm d\mathfrak z^i\wedge\mathrm d\overline{\mathfrak z}^j,\qquad \tau_{ij}(q)\coloneqq-\frac{\partial\mathfrak w_j}{\partial\mathfrak z^i}(q)=\int_{\mathfrak b_{j,q}}\omega_{i,q},
\]
where $\omega_{i,q}$ is defined in Lemma~\ref{lem:Bd-dual-forms}. As $\mathcal B^{\mathrm{reg}}\ni q\to q_0$, the coefficients satisfy
\[
\begin{aligned}
\textnormal{(1)}\quad&\operatorname{Im}(\tau_{ij}(q))\longrightarrow\operatorname{Im}(\tau_{ij}(q_0))<\infty,&& (i,j)\notin\{(k,k)\mid k\in B^c\},\\
\textnormal{(2)}\quad&\operatorname{Im}(\tau_{kk}(q))\asymp-\log|\mathfrak z^k(q)|,&& k\in B^c.
\end{aligned}
\]
Here $\asymp$ is defined in Subsection~\ref{subsec:intro-main}.
\end{theorem}

\begin{proof}
Baraglia--Huang \cite[Section~2]{baraglia2017special} and Mazzeo--Swoboda--Weiss--Witt \cite[Section~2.3]{mazzeo2019asymptotic} derived the integral expression for $\tau_{ij}$. We use this expression directly to analyze $\operatorname{Im}(\tau_{ij})$.

Suppose that $(i,j)\notin\{(k,k)\mid k\in B^c\}$. For $q$ in a sufficiently small simply connected open subset $\mathcal V\subset\mathcal U\cap\mathcal B^{\mathrm{reg}}$ approaching $q_0$, choose a family of deformations $\{\mathfrak b_{j,q}\}_{q\in\mathcal V}$ of $\mathfrak b_{j,q_0}$ that varies continuously with $q$. By Lemma~\ref{lem:family-normalization}, the normalized differentials $\omega_{i,q}$ also depend continuously on $q$. The limiting integral is finite: if $i\in B$, then $\omega_{i,q_0}$ is holomorphic, while if $i\in B^c$ and $j\neq i$, then $\mathfrak b_{j,q_0}$ does not join the two poles of $\omega_{i,q_0}$. Since $\operatorname{Im}(\tau_{ij})$ is independent of the choice of deformation, these local extensions agree when $\mathcal V$ and the continuous family of deformations are changed. Hence $\operatorname{Im}(\tau_{ij})$ extends continuously to $q_0$, proving (1).

We now prove (2). Fix $1\leq k\leq d$, and let $\{Q_{i,q}\}_{i=1}^{3g-3}$ be the basis of $H^0(C,K_C^2)$ corresponding to $\{\omega_{i,q}\}_{i=1}^{3g-3}$. Choose a coordinate $z$ on $\mathbb B_{2k-1}$ such that $q_0=z^2\,\mathrm dz^2$. Write $q\in\mathcal U$ as
\[
q=(z-z_1(q))(z-z_2(q))g_q(z)\,\mathrm dz^2,\qquad Q_{g+k,q}=f_{g+k,q}(z)\,\mathrm dz^2,
\]
and set $\delta_q\coloneqq z_1(q)-z_2(q)$. Here $g_q$ and $f_{g+k,q}$ are holomorphic in $z$ and depend continuously on $q$. Near the two points of $\widetilde\pi_{q_0}^{-1}(p_{2k-1,q_0})$, the form $\omega_{g+k,q_0}$ is locally of the form $\displaystyle \pm f_{g+k,q_0}(\widetilde z)\frac{\mathrm d\widetilde z}{2\widetilde z}$. The residue normalization in Lemma~\ref{lem:Bd-dual-forms} therefore gives $f_{g+k,q_0}(0)=(\pi\mathrm i)^{-1}$.

The imaginary part $\operatorname{Im}\tau_{g+k,g+k}(q)$ is given by $\operatorname{Im}\int_{\mathfrak b_{g+k,q}}\omega_{g+k,q}$. The contribution from the portion of $\mathfrak b_{g+k,q}$ outside $\widetilde\pi_q^{-1}(\mathbb B_{2k-1})$ remains uniformly bounded as $q\to q_0$, since this portion remains uniformly away from the poles and the integrand depends continuously on $q$ there. It suffices to estimate the integral inside $\widetilde\pi_q^{-1}(\mathbb B_{2k-1})$.

We may choose its local projection to be a ray $z(t)=z_1(q)+te^{\mathrm i\theta_q}$, $0\leq t\leq r$, which does not pass through $z_2(q)$, since $\operatorname{Im}\tau_{g+k,g+k}(q)$ is independent of the choice of deformation $\mathfrak b_{g+k,q}$. By \eqref{eq:H0-isomorphism}, the local contribution to $\tau_{g+k,g+k}(q)$ inside $\widetilde\pi_q^{-1}(\mathbb B_{2k-1})$ is
\[
\int_{\mathfrak b_{g+k,q}\cap \widetilde\pi_q^{-1}(\mathbb B_{2k-1})}\omega_{g+k,q} = \int_{\ell_{2k,q}\cap\mathbb B_{2k-1}}\frac{f_{g+k,q}(z)\,\mathrm dz^2}{\sqrt{q}} = -\int_{z(t)}\frac{f_{g+k,q}(z)\,\mathrm dz}{\sqrt{(z-z_1(q))(z-z_2(q))g_q(z)}}.
\]
By Lemma~\ref{lem:quadratic-vanishing} and Lemma~\ref{lem:local-logarithmic-estimate} and the uniformly bounded contribution outside $\widetilde\pi_q^{-1}(\mathbb B_{2k-1})$, we have
\[
\operatorname{Im}(\tau_{g+k,g+k}(q))\asymp-\log|\delta_q|,\qquad |\mathfrak z^{g+k}(q)|\asymp|\delta_q|^2.
\]
Consequently, $\operatorname{Im}(\tau_{g+k,g+k}(q))\asymp-\log|\mathfrak z^{g+k}(q)|$, which proves (2).
\end{proof}

Theorem~\ref{thm:Bd-singular-model} describes the asymptotic behavior of the coefficient matrix of the special K\"ahler metric on $\mathcal U\cap\mathcal B^{\mathrm{reg}}$. More precisely, all the coefficients $\operatorname{Im}(\tau_{ij}(q))$ extend continuously to $q_0$, except for the diagonal terms $\operatorname{Im}(\tau_{kk}(q))$, $k\in B^c$, which have logarithmic divergence. This motivates the following singular model for $\omega_{\mathrm{SK}}$:
\[
\omega_{\mathrm{mod}}\coloneqq\frac{\mathrm i}{2}\left(\sum_{i\in B}\mathrm d\mathfrak z^i\wedge\mathrm d\overline{\mathfrak z}^i+\sum_{k\in B^c}-\log|\mathfrak z^k|\,\mathrm d\mathfrak z^k\wedge\mathrm d\overline{\mathfrak z}^k\right).
\]
After shrinking $\mathcal U$ if necessary, $\omega_{\mathrm{SK}}(q)\asymp\omega_{\mathrm{mod}}(q)$ as positive Hermitian forms for $q\in\mathcal U\cap\mathcal B^{\mathrm{reg}}$.

\section{The Stratum $\mathcal B_{2g-2}$}\label{sec:B2g-2}

In this section, we investigate the asymptotic behavior of the special K\"ahler metric $\omega_{\mathrm{SK}}$ in the vicinity of the stratum $\mathcal B_{2g-2}$. Fix $q_0\in\mathcal B_{2g-2}$ and write its divisor as
\[
\operatorname{div}(q_0)=2p_{1,q_0}+2p_{3,q_0}+\cdots+2p_{4g-5,q_0}.
\]
Let $\mathcal U$ be a sufficiently small simply connected neighborhood of $q_0$ in $\mathcal B$. We will analyze $\omega_{\mathrm{SK}}$ on $\mathcal U\cap\mathcal B^{\mathrm{reg}}$ by implementing the steps outlined in Subsection~\ref{subsec:intro-strategy}. The procedure is analogous to the case $q_0\in\mathcal B_d$ discussed above, but for $q_0\in\mathcal B_{2g-2}$, the normalized spectral cover $\widetilde\pi_{q_0}\colon\widetilde\Sigma_{q_0}\to C$ is an unbranched double cover. Hence we cannot take a trivializing disc for $\widetilde\pi_{q_0}$ as in the case $q_0\in\mathcal B_d$. Due to this topological difference, we need to redesign the adapted paths on $\widetilde\Sigma_{q_0}$. Note that $\widetilde\Sigma_{q_0}$ has genus $2g-1$, and $H_1\big(\widetilde\Sigma_{q_0},\mathbb R\big)^-$ has dimension $2g-2$.

\subsection{Adapted paths and deformations}\label{subsec:B2g-2-adapted-paths}

The construction of adapted paths begins with a symplectic basis of $H_1(C,\mathbb Z)$ compatible with the monodromy of the unbranched double cover $\widetilde\pi_{q_0}$. Let $\phi\colon H_1(C,\mathbb Z)\to\mathbb Z_2$ denote its monodromy homomorphism. We first choose such a basis as follows.

\begin{lemma}\label{lem:B2g-2-monodromy-basis}
There exists a symplectic basis $\{\alpha_i,\beta_i\}_{i=1}^g$ of $H_1(C,\mathbb Z)$ such that
\[
\phi(\beta_1)=1,\qquad \phi(\alpha_i)=0\quad\text{for }1\leq i\leq g,\qquad \phi(\beta_i)=0\quad\text{for }2\leq i\leq g.
\]
\end{lemma}

\begin{proof}
Since $q_0$ admits no global square root, the monodromy homomorphism $\phi$ is nontrivial. Choose a primitive class $\beta_1$ with $\phi(\beta_1)=1$ and extend it to a symplectic basis $\{\alpha_i,\beta_i\}_{i=1}^g$ of $H_1(C,\mathbb Z)$. For each $i\geq2$, successively perform the following elementary symplectic transformations whenever necessary:
\[
\begin{cases}
(\alpha_i,\alpha_1)\mapsto(\alpha_i+\beta_1,\alpha_1+\beta_i), & \text{if } \phi(\alpha_i)=1,\\
(\beta_i,\alpha_1)\mapsto(\beta_i+\beta_1,\alpha_1-\alpha_i), & \text{if } \phi(\beta_i)=1.
\end{cases}
\]
After these transformations, $\phi(\alpha_i)=\phi(\beta_i)=0$ for every $i\geq2$, while $\phi(\beta_1)=1$. Finally, if $\phi(\alpha_1)=1$, replace $\alpha_1$ by $\alpha_1+\beta_1$. The resulting symplectic basis has the required properties.
\end{proof}

The symplectic basis $\{\alpha_i,\beta_i\}_{i=1}^g$ given by Lemma~\ref{lem:B2g-2-monodromy-basis} allows us to describe the topology of $\widetilde\Sigma_{q_0}$, as illustrated in Figure~\ref{fig:B2g-2-double-cover}. For each $1\leq i\leq g$ and $2\leq j\leq g$, the preimage of each $\alpha_i$ and $\beta_j$ consists of two disjoint oriented loops on $\widetilde\Sigma_{q_0}$, denoted by $\left\{a_{i,q_0}^{(1)}, a_{i,q_0}^{(2)}\right\}$ and $\left\{b_{j,q_0}^{(1)}, b_{j,q_0}^{(2)}\right\}$, respectively. By contrast, the preimage of $\beta_1$ is a single loop, denoted by $b_{1,q_0}$. The configuration of $a_{1,q_0}^{(1)}$, $a_{1,q_0}^{(2)}$, and $b_{1,q_0}$ is shown in Figure~\ref{fig:B2g-2-first-handle}.

\begin{figure}[!t]
\centering
\begin{subfigures}
    \begin{minipage}[b]{0.58\linewidth}
        \centering
        \includegraphics[width=\linewidth]{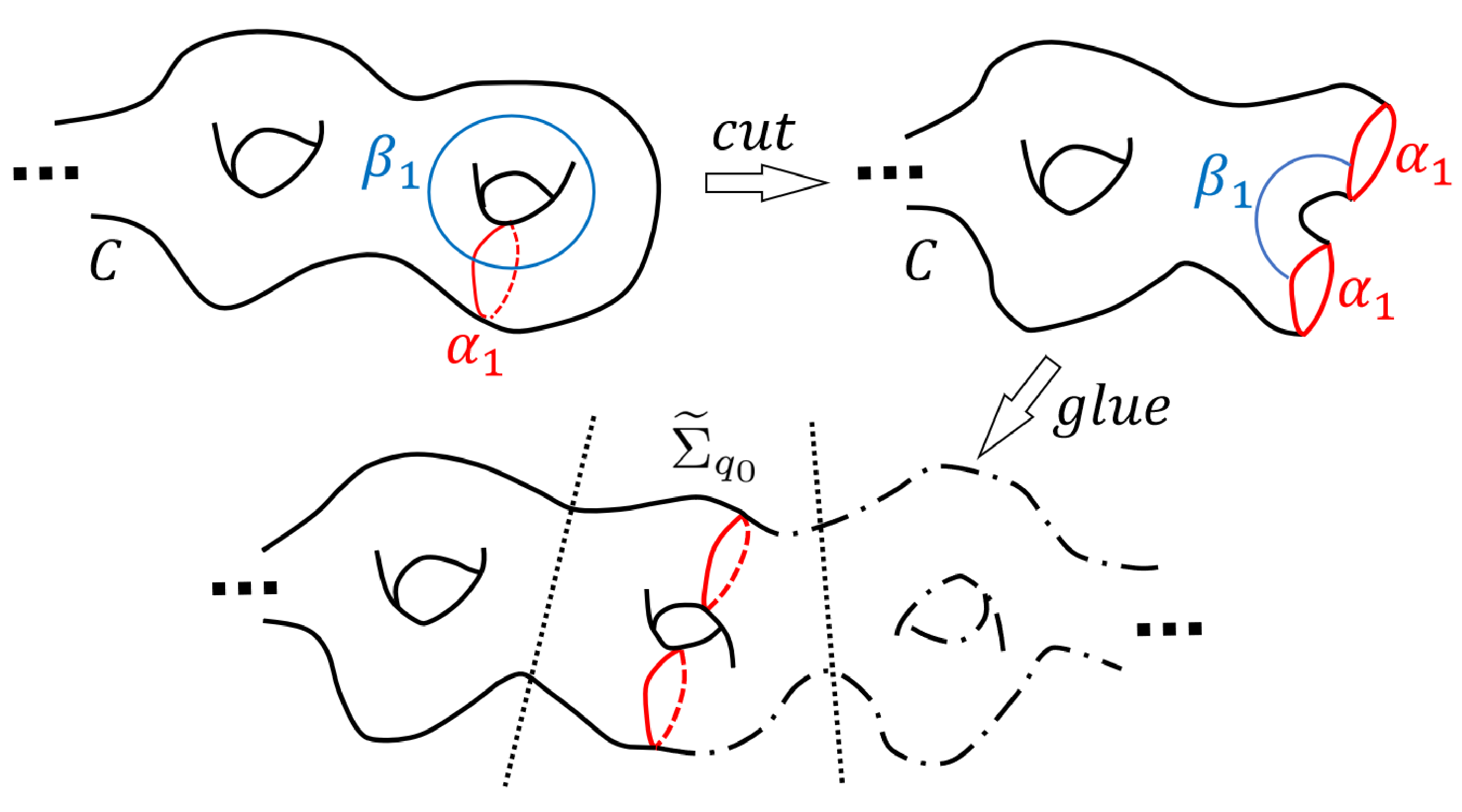}
        \figalttext[Construction of the unbranched double cover]
        {Schematic of the unbranched double cover. The cycle $\alpha_1$ is shown in red, while $\beta_1$ is shown in blue across the two cross-glued copies.}
        \caption{The unbranched double cover $\widetilde\pi_{q_0}\colon\widetilde\Sigma_{q_0}\to C$, obtained by cutting $C$ along $\alpha_1$, taking two copies of the resulting surface, and cross-gluing the corresponding boundary components.}
        \label{fig:B2g-2-double-cover}
    \end{minipage}
    \hfill
    \begin{minipage}[b]{0.39\linewidth}
        \centering
        \includegraphics[width=\linewidth]{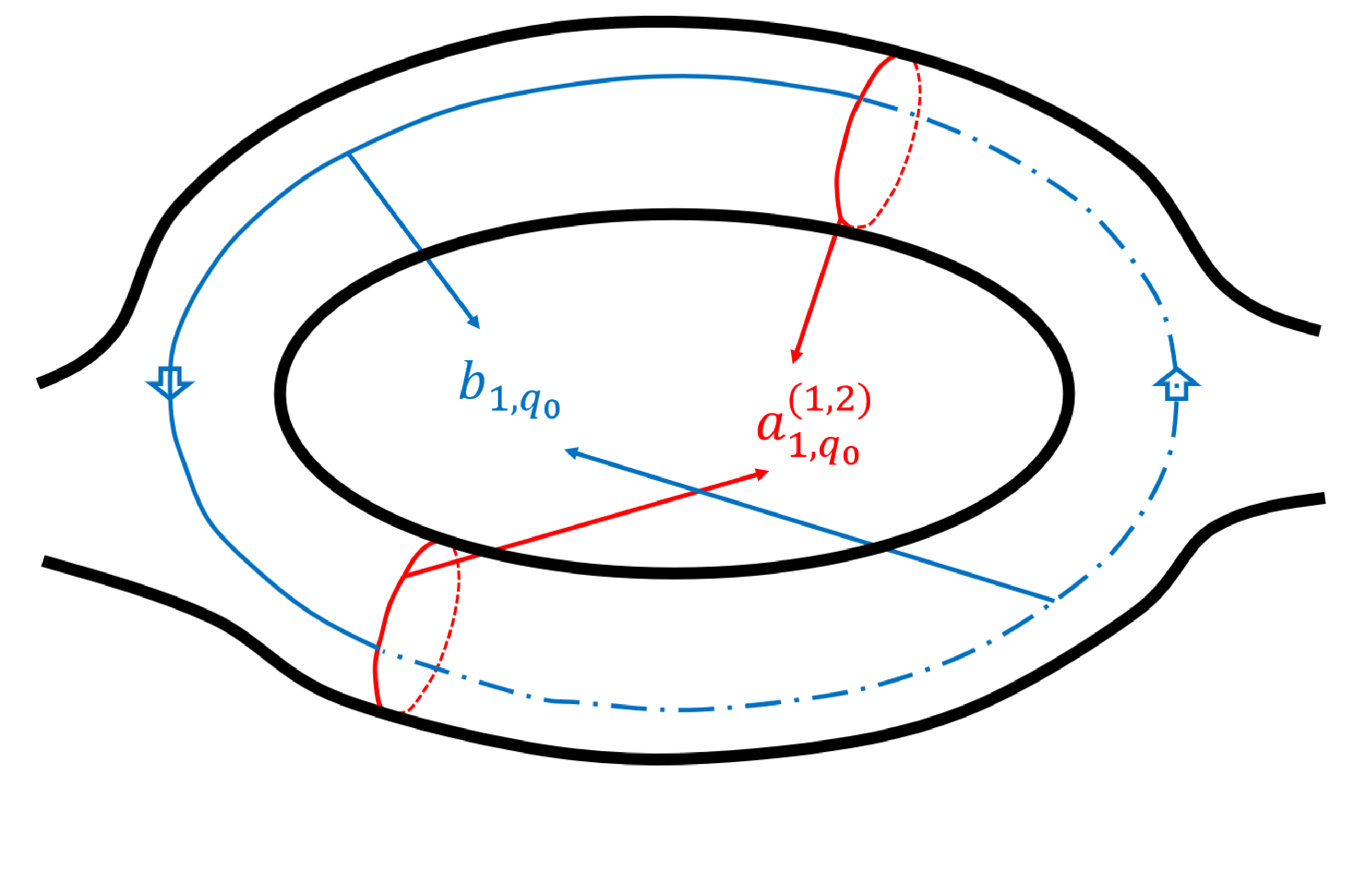}
        \figalttext[Lifts of the first handle cycles]
        {Magnified view of the lower region between the dashed arcs. The lifts of $\alpha_1$ are shown in red, and the lifts of $\beta_1$ are shown in blue.}
        \caption{A magnified view of the region cut out by the two dashed arcs in the lower part of Figure~\ref{fig:B2g-2-double-cover}, showing the lifts of $\alpha_1$ and $\beta_1$.}
        \label{fig:B2g-2-first-handle}
    \end{minipage}
\end{subfigures}
\end{figure}

\begin{lemma}\label{lem:B2g-2-adapted-paths}
There exist oriented paths $\{\mathfrak a_{i,q_0},\mathfrak b_{i,q_0}\}_{i=2}^{3g-2}$ on $\widetilde\Sigma_{q_0}$ such that
\[
\langle\mathfrak a_{i,q_0},\mathfrak b_{j,q_0}\rangle=\delta_{ij},
\qquad
\langle\mathfrak a_{i,q_0},\mathfrak a_{j,q_0}\rangle=\langle\mathfrak b_{i,q_0},\mathfrak b_{j,q_0}\rangle=0,\qquad 2\leq i,j\leq3g-2.
\]
Moreover, $\{\mathfrak a_{i,q_0},\mathfrak b_{i,q_0}\}_{i=2}^{g}$ forms a symplectic basis of $H_1(\widetilde\Sigma_{q_0},\mathbb R)^-$.
\end{lemma}

\begin{proof}
For $2\leq i\leq g$, define
\begin{equation}\label{eq:B2g-2-tangential-cycles}
\mathfrak a_{i,q_0}\coloneqq\frac{1}{\sqrt2}\left(a_{i,q_0}^{(1)}-a_{i,q_0}^{(2)}\right),\qquad \mathfrak b_{i,q_0}\coloneqq\frac{1}{\sqrt2}\left(b_{i,q_0}^{(1)}-b_{i,q_0}^{(2)}\right).
\end{equation}
These paths form a symplectic basis of $H_1\big(\widetilde\Sigma_{q_0},\mathbb R\big)^-$.

For each $1\leq k\leq2g-2$, choose pairwise disjoint sufficiently small discs $\mathbb B_{2k-1}\subset C$ centered at $p_{2k-1,q_0}$ and set $\ell_{2k-1,q_0}\coloneqq\partial\mathbb B_{2k-1}$. Choose simultaneously loops $\{\ell_{2k,q_0}\}_{k=1}^{2g-2}$ on $C$, where each $\ell_{2k,q_0}$ passes through $p_{2k-1,q_0}$, is disjoint from $\{\alpha_i,\beta_i\}_{i=2}^g$, is freely homotopic to $\beta_1$, and satisfies the following conditions.
\begin{itemize}
\item The loops in $\{\ell_{2k,q_0}\}_{k=1}^{2g-2}$ are pairwise disjoint.
\item The loop $\ell_{2k,q_0}$ is disjoint from $\ell_{2j-1,q_0}$ whenever $j\neq k$, and meets $\ell_{2k-1,q_0}$ transversely at exactly two points, corresponding to its entry into and exit from $\mathbb B_{2k-1}$.
\end{itemize}

The preimage of $\ell_{2k-1,q_0}$ consists of two disjoint contractible loops; choose one and denote it by $\mathfrak a_{g+k,q_0}$. The preimage of $\ell_{2k,q_0}$ is a single loop, which is divided by the two points of $\widetilde\pi_{q_0}^{-1}(p_{2k-1,q_0})$ into two oriented arcs $b_{g+k,q_0}^{(1)}$ and $b_{g+k,q_0}^{(2)}$. Set $\displaystyle\mathfrak b_{g+k,q_0}\coloneqq\frac12\left(b_{g+k,q_0}^{(1)}-b_{g+k,q_0}^{(2)}\right)$. After choosing the orientation of $\mathfrak a_{g+k,q_0}$ suitably, a local intersection computation gives $\langle\mathfrak a_{g+k,q_0},\mathfrak b_{g+k,q_0}\rangle=1$. All other intersections vanish by construction, which proves the result.
\end{proof}

The configuration of $\mathfrak a_{g+k,q_0}$, $b_{g+k,q_0}^{(1)}$, and $b_{g+k,q_0}^{(2)}$, together with its deformation to a nearby normalized spectral curve, is shown in Figures~\ref{fig:B2g-2-normal-paths-q0} and~\ref{fig:B2g-2-normal-paths-q}. Next, we deform $\{\mathfrak a_{i,q_0},\mathfrak b_{i,q_0}\}_{i=2}^{3g-2}$ to $\widetilde\Sigma_q$ for $q\in\mathcal U$. After shrinking $\mathcal U$, each disc $\mathbb B_{2k-1}$ contains either one double zero or two simple zeros of $q$ for every $q\in\mathcal U$.

\begin{figure}[!t]
\centering
\begin{subfigures}
    \begin{minipage}[b]{0.48\linewidth}
        \centering
        \includegraphics[width=\linewidth]{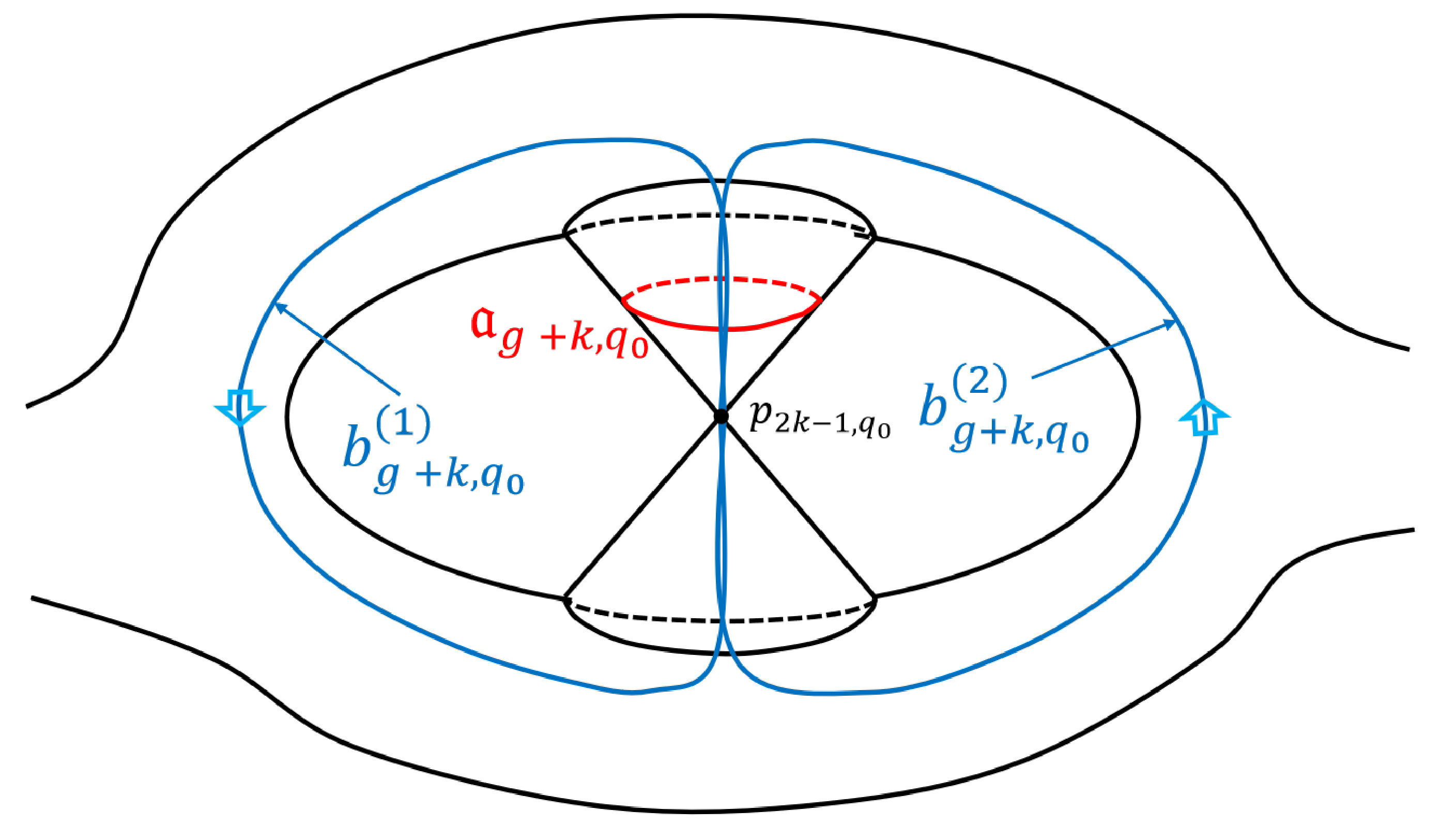}
        \figalttext[Normal paths on the limiting spectral curve]
        {Limiting configuration of the adapted paths. The loop $\mathfrak a_{g+k,q_0}$ is shown in red, while the two oriented lifts $b_{g+k,q_0}^{(1)}$ and $b_{g+k,q_0}^{(2)}$ are shown in blue.}
        \caption{A magnified view of the region cut out by the two dashed arcs in the lower part of Figure~\ref{fig:B2g-2-double-cover}, showing $\mathfrak a_{g+k,q_0}$, $b_{g+k,q_0}^{(1)}$, and $b_{g+k,q_0}^{(2)}$.}
        \label{fig:B2g-2-normal-paths-q0}
    \end{minipage}
    \hfill
    \begin{minipage}[b]{0.48\linewidth}
        \centering
        \includegraphics[width=\linewidth]{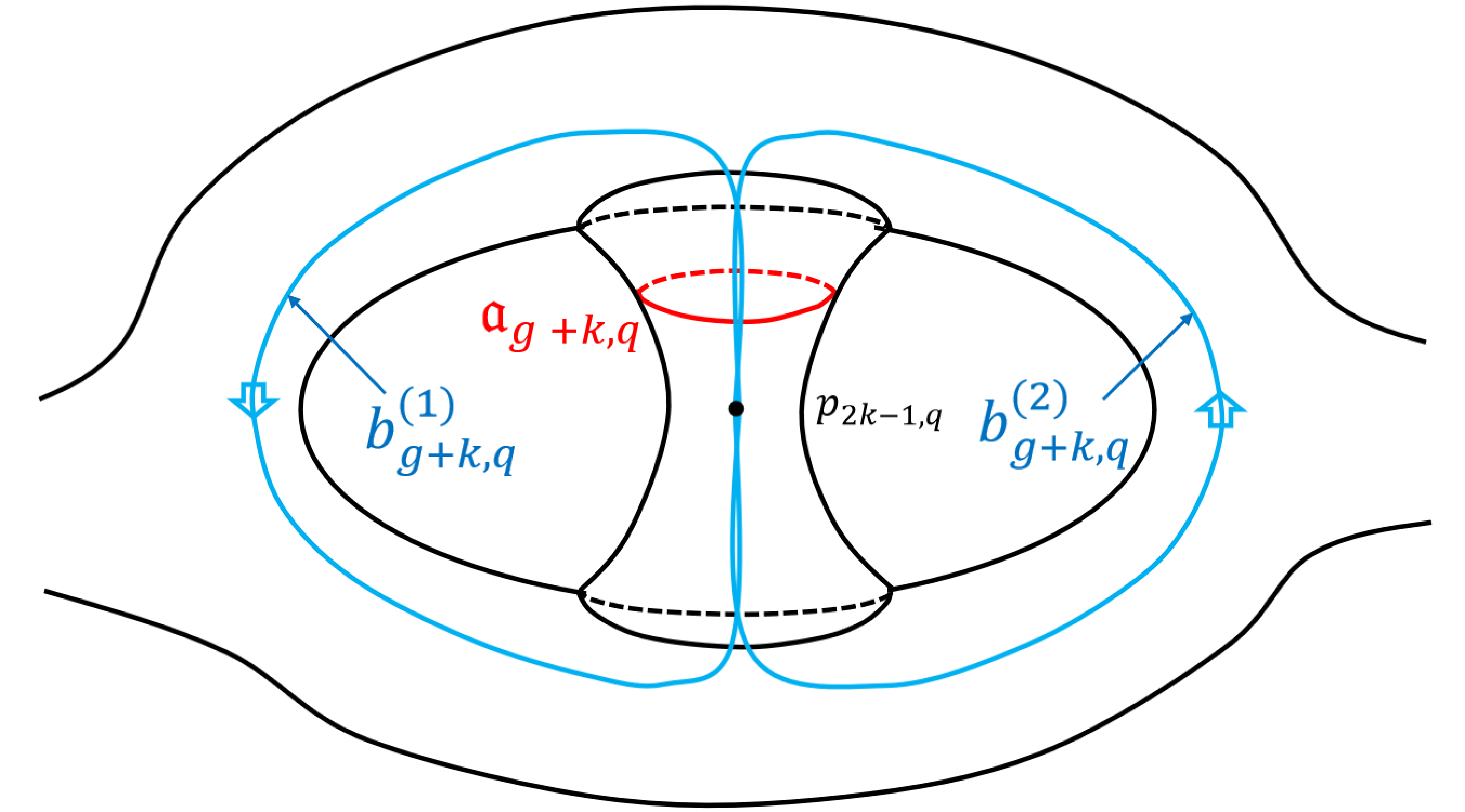}
        \figalttext[Deformed normal paths on a nearby spectral curve]
        {Deformed configuration of the adapted paths. The loop $\mathfrak a_{g+k,q}$ is shown in red, while the two lifted paths $b_{g+k,q}^{(1)}$ and $b_{g+k,q}^{(2)}$ are shown in blue on the nearby normalized spectral curve.}
        \caption{Configuration of the deformed paths $\mathfrak a_{g+k,q}$, $b_{g+k,q}^{(1)}$, and $b_{g+k,q}^{(2)}$ on the normalized spectral curve $\widetilde\Sigma_q$ for $q\in\mathcal U$.}
        \label{fig:B2g-2-normal-paths-q}
    \end{minipage}
\end{subfigures}
\end{figure}

For $2\leq i\leq g$, denote the preimages of $\alpha_i$ and $\beta_i$ under $\widetilde\pi_q$ by $\left\{a_{i,q}^{(1)},a_{i,q}^{(2)}\right\}$ and $\left\{b_{i,q}^{(1)},b_{i,q}^{(2)}\right\}$, respectively, and define $\{\mathfrak a_{i,q},\mathfrak b_{i,q}\}_{i=2}^{g}$ by \eqref{eq:B2g-2-tangential-cycles}.

For $1\leq k\leq2g-2$, the preimage of $\ell_{2k-1,q}\coloneqq\partial\mathbb B_{2k-1}$ under $\widetilde\pi_q$ consists of two disjoint loops. Let $\mathfrak a_{g+k,q}$ be the one deforming continuously from $\mathfrak a_{g+k,q_0}$. Note that both endpoints of the auxiliary path $\ell_{2k,q_0}$ coincide at the double zero $p_{2k-1,q_0}$. We deform $\ell_{2k,q_0}$ according to Type~2 inside $\mathbb B_{2k-1}$, as described in Subsection~\ref{subsec:Bd-deformation}, while keeping the path fixed outside $\mathbb B_{2k-1}$. Denote by $(\ell_{2k,q})$ the collection of all such deformations. For a choice $\ell_{2k,q}\in(\ell_{2k,q})$, let $b_{g+k,q}^{(1)}$ and $b_{g+k,q}^{(2)}$ be the corresponding lifted paths, and define $\displaystyle \mathfrak b_{g+k,q}\coloneqq\frac12\left(b_{g+k,q}^{(1)}-b_{g+k,q}^{(2)}\right)$. When $q$ has two simple zeros in $\mathbb B_{2k-1}$, any two choices satisfy
\begin{equation}\label{eq:B2g-2-b-multivalued}
\left[\mathfrak b_{g+k,q}-\mathfrak b'_{g+k,q}\right] = n_k\left[\mathfrak a_{g+k,q}\right],\qquad n_k\in\mathbb Z.
\end{equation}

By an argument similar to that in Lemma~\ref{lem:Bd-deformation}, for each $q\in\mathcal U\cap\mathcal B^{\mathrm{reg}}$, there exists a neighborhood $\mathcal V\subset\mathcal U\cap\mathcal B^{\mathrm{reg}}$ of $q$ on which the paths $\mathfrak b_{g+k,q'}$ can be chosen continuously. With suitable orientations, the collection $\{\mathfrak a_{i,q'},\mathfrak b_{i,q'}\}_{i=2}^{3g-2}$ forms a symplectic basis of $H_1\big(\widetilde\Sigma_{q'},\mathbb R\big)^-$ for every $q'\in\mathcal V$.

We now define the integral functions
\begin{equation}\label{eq:B2g-2-period-coordinates}
\mathfrak z^i(q)\coloneqq\int_{\mathfrak a_{i,q}}\vartheta_q,
\qquad
\mathfrak w_i(q)\coloneqq-\int_{\mathfrak b_{i,q}}\vartheta_q,
\qquad 2\leq i\leq3g-2.
\end{equation}
The functions $\{\mathfrak z^k\}_{k=2}^{3g-2}\cup\{\mathfrak w_k\}_{k=2}^{g}$ are single-valued on $\mathcal U$, whereas $\{\mathfrak w_{g+k}\}_{k=1}^{2g-2}$ are generally multivalued. By \eqref{eq:B2g-2-b-multivalued}, any two local branches satisfy
\begin{equation}\label{eq:B2g-2-w-multivalued}
\mathfrak w_{g+k}(q)=\mathfrak w'_{g+k}(q)-n_k\mathfrak z^{g+k}(q),\qquad n_k\in\mathbb Z,\qquad 1\leq k\leq2g-2.
\end{equation}

\subsection{Singular model of the special K\"ahler metric near $\mathcal B_{2g-2}$}\label{subsec:B2g-2-singular-model}

We now use the integral functions defined in \eqref{eq:B2g-2-period-coordinates} to characterize $\mathcal B_{2g-2}$ locally and describe the singular behavior of $\omega_{\mathrm{SK}}$ on $\mathcal U\cap\mathcal B^{\mathrm{reg}}$. For convenience, set
\[
B\coloneqq\{2,\ldots,g\},\qquad B^c\coloneqq\{g+1,\ldots,3g-2\}.
\]

\begin{theorem}\label{thm:B2g-2-zero-locus}
Let $D_{q_0}\coloneqq\frac12\operatorname{div}(q_0)=\sum_{k=1}^{2g-2}p_{2k-1,q_0}$. After shrinking $\mathcal U$ if necessary,
\[
\mathcal B_{2g-2}\cap\mathcal U = \bigl\{ q\in\mathcal U \mid \mathfrak z^k(q)=0\text{ for all }k\in B^c\bigr\}.
\]
In particular, $\mathcal B_{2g-2}\cap\mathcal U$ is a complex submanifold of dimension $g-1$, and, under the identification $T_{q_0}\mathcal B\cong H^0(C,K_C^2)$, we have $T_{q_0}\mathcal B_{2g-2}= H^0\bigl(C,K_C^2(-D_{q_0})\bigr)$.
\end{theorem}

\begin{proof}
By Lemma~\ref{lem:quadratic-vanishing}, $\mathfrak z^{g+k}(q)=0$ if and only if $q$ has a double zero in $\mathbb B_{2k-1}$. After shrinking $\mathcal U$ so that $\mathcal U\cap\mathcal B_{\mathrm{ab}}=\varnothing$, these conditions are equivalent to $q\in\mathcal B_{2g-2}$, proving the zero-locus description.

By Lemma~\ref{lem:Bd-vanishing-differential}, $\mathrm d\mathfrak z^{g+k}|_{q_0}(\dot q)=0$ if and only if $\dot q$ vanishes at $p_{2k-1,q_0}$ for each $1\leq k\leq2g-2$. We now argue as in the proof of Theorem~\ref{thm:Bd-zero-locus}. Since $q_0$ has no global square root, $K_C^{-1}(D_{q_0})$ is a nontrivial degree-zero line bundle, and hence the Riemann--Roch theorem gives
\[
\dim H^0\left(C,K_C^2\left(-\sum_{j\neq k}p_{2j-1,q_0}\right)\right) = \dim H^0\left(C,K_C^2(-D_{q_0})\right)+1.
\]
Thus there exists $\dot q_k\in H^0(C,K_C^2)$ vanishing at $p_{2j-1,q_0}$ for every $j\neq k$, but not at $p_{2k-1,q_0}$. After rescaling, $\mathrm d\mathfrak z^{g+j}|_{q_0}(\dot q_k)=\delta_{jk}$. Thus the differentials $\{\mathrm d\mathfrak z^k\}_{k\in B^c}$ are linearly independent at $q_0$.
\end{proof}

As in the previous section, we introduce normalized meromorphic differentials on the normalized spectral curves in order to express the coefficients $\operatorname{Im}(\tau_{ij})$ of $\omega_{\mathrm{SK}}$.

\begin{lemma}\label{lem:B2g-2-dual-forms}
For every $q\in\mathcal U$, there exists a unique basis $\{\omega_{i,q}\}_{i=2}^{3g-2}$ of $H^0\bigl(\widetilde\Sigma_q,K_{\widetilde\Sigma_q}(\widetilde D_q)\bigr)^-$ satisfying
\[
\int_{\mathfrak a_{i,q}}\omega_{j,q}=\delta_{ij},\qquad 2\leq i,j\leq3g-2.
\]
At $q=q_0$, the forms $\omega_{i,q_0}$ are holomorphic for $2\leq i\leq g$, while, for each $1\leq k\leq2g-2$, the form $\omega_{g+k,q_0}$ has simple poles only at the two points of $\widetilde\pi_{q_0}^{-1}(p_{2k-1,q_0})$, with residues $\pm(2\pi\mathrm i)^{-1}$.
\end{lemma}

\begin{proof}
For $q\in\mathcal U$, let $I_q$ be the set of indices $k\in\{1,\ldots,2g-2\}$ for which $q$ has a double zero in $\mathbb B_{2k-1}$. Then $\mathfrak a_{g+k,q}$ is contractible for $k\in I_q$, while the remaining pairs form a symplectic basis of $H_1(\widetilde\Sigma_q,\mathbb R)^-$. The assertion follows from Lemma~\ref{lem:adapted-normalized-differentials}.
\end{proof}

Using $\Psi_q$ in \eqref{eq:H0-isomorphism}, set $Q_{i,q}\coloneqq\Psi_q^{-1}(\omega_{i,q})\in H^0(C,K_C^2)$. Since the projected $\mathfrak a$-paths are fixed and disjoint from the zeros of every $q\in\mathcal U$, the argument of Lemma~\ref{lem:family-normalization} applies verbatim and shows that $Q_{i,q}$ depends continuously on $q$. As before, continuity of $\omega_{i,q}$ will mean continuity of $Q_{i,q}$ in the fixed vector space $H^0(C,K_C^2)$.

The normalized differentials above give the integral representation of the matrix $(\tau_{ij})$. Its asymptotic behavior is summarized in the following theorem.

\begin{theorem}\label{thm:B2g-2-singular-model}
For $q\in\mathcal U\cap\mathcal B^{\mathrm{reg}}$,
\[
\omega_{\mathrm{SK}} = \frac{\mathrm i}{2}\sum_{i,j=2}^{3g-2}\operatorname{Im}(\tau_{ij})\,\mathrm d\mathfrak z^i\wedge\mathrm d\overline{\mathfrak z^j}, \qquad\tau_{ij}(q)\coloneqq-\frac{\partial\mathfrak w_j}{\partial\mathfrak z^i}(q) = \int_{\mathfrak b_{j,q}}\omega_{i,q}.
\]
As $\mathcal B^{\mathrm{reg}}\ni q\to q_0$,
\[
\begin{aligned}
\textnormal{(1)}\quad&\operatorname{Im}(\tau_{ij}(q))\longrightarrow\operatorname{Im}(\tau_{ij}(q_0))<\infty,&& (i,j)\notin\{(k,k)\mid k\in B^c\},\\
\textnormal{(2)}\quad&\operatorname{Im}(\tau_{kk}(q))\asymp-\log|\mathfrak z^k(q)|,&& k\in B^c.
\end{aligned}
\]
\end{theorem}

\begin{proof}

Suppose that $(i,j)\notin\{(k,k)\mid k\in B^c\}$. On a sufficiently small simply connected open subset of $\mathcal U\cap\mathcal B^{\mathrm{reg}}$ approaching $q_0$, choose $\{\mathfrak b_{j,q}\}$ so that it varies continuously with $q$. The normalized differentials $\omega_{i,q}$ also depend continuously on $q$. If $i\in B$, then $\omega_{i,q_0}$ is holomorphic, while if $i\in B^c$ and $j\neq i$, then $\mathfrak b_{j,q_0}$ does not join the two poles of $\omega_{i,q_0}$. Thus the corresponding integral remains finite and varies continuously up to $q=q_0$. Since $\operatorname{Im}(\tau_{ij})$ is independent of the choice of deformation, this proves (1).

For $k = g + l \in B^c$, choose a coordinate $z$ on $\mathbb B_{2l-1}$ such that
\[
q_0=z^2\,\mathrm dz^2,\qquad q=(z-z_1(q))(z-z_2(q))g_q(z)\,\mathrm dz^2,
\]
and set $\delta_q\coloneqq z_1(q)-z_2(q)$. The same local argument as in the proof of Theorem~\ref{thm:Bd-singular-model}, together with Lemmas~\ref{lem:local-logarithmic-estimate} and \ref{lem:quadratic-vanishing}, gives
\[
\operatorname{Im}(\tau_{kk}(q))\asymp-\log|\delta_q|,\qquad |\mathfrak z^k(q)|\asymp|\delta_q|^2.
\]
Combining these estimates proves (2).
\end{proof}

\section{The Split Stratum $\mathcal B_{\mathrm{ab}}$}\label{sec:Bab}

Fix $q_0=\psi^2\in\mathcal B_{\mathrm{ab}}$ and write
\[
\operatorname{div}(q_0)=2p_{1,q_0}+\cdots+2p_{4g-5,q_0},\qquad D_{q_0}\coloneqq\operatorname{div}(\psi)=\sum_{k=1}^{2g-2}p_{2k-1,q_0}.
\]
The normalized spectral curve is disconnected: $\widetilde\Sigma_{q_0}=C_1\sqcup C_2$, with the deck transformation $\sigma_{q_0}$ exchanging the two components. Let $\mathcal U$ be a sufficiently small simply connected neighborhood of $q_0$ in $\mathcal B$. Following the strategy of Subsection~\ref{subsec:intro-strategy}, we construct adapted paths on the split normalization and use their deformations to determine the singular behavior of $\omega_{\mathrm{SK}}$ on $\mathcal U\cap\mathcal B^{\mathrm{reg}}$.

\subsection{Adapted paths and deformations}\label{subsec:Bab-adapted-paths}

We will construct a system of paths $\{\mathfrak a_{i,q_0},\mathfrak b_{i,q_0}\}_{i=1}^{3g-3}$ on $\widetilde\Sigma_{q_0}$ and then deform them to $\widetilde\Sigma_q$ for $q\in\mathcal U$. Let $\mathbb D\subset C$ be an embedded disc containing $\operatorname{Zero}(q_0)$, and choose a symplectic basis $\{\alpha_i,\beta_i\}_{i=1}^g$ of $H_1(C,\mathbb Z)$ disjoint from $\mathbb D$. Denote the lifts of $\alpha_i$ and $\beta_i$ to $C_1$ and $C_2$ by $\left\{a_{i,q_0}^{(1)},a_{i,q_0}^{(2)}\right\}$ and $\left\{b_{i,q_0}^{(1)},b_{i,q_0}^{(2)}\right\}$, respectively, and define
\begin{equation}\label{eq:Bab-tangential-cycles}
\mathfrak a_{i,q_0}\coloneqq\frac{1}{\sqrt2}\left(a_{i,q_0}^{(1)}-a_{i,q_0}^{(2)}\right),\qquad
\mathfrak b_{i,q_0}\coloneqq\frac{1}{\sqrt2}\left(b_{i,q_0}^{(1)}-b_{i,q_0}^{(2)}\right),\quad 1\leq i\leq g.
\end{equation}
These paths form a symplectic basis of $H_1\big(\widetilde\Sigma_{q_0},\mathbb R\big)^-$.

To construct the remaining paths $\{\mathfrak a_{i,q_0},\mathfrak b_{i,q_0}\}_{i=g+1}^{3g-3}$, we introduce auxiliary paths $\{\ell_{k,q_0}\}_{k=1}^{4g-5}$ on $C$. For $1\leq k\leq2g-2$, choose a sufficiently small disc $\mathbb B_{2k-1}\subset\mathbb D$ centered at $p_{2k-1,q_0}$ and set $\ell_{2k-1,q_0}\coloneqq\partial\mathbb B_{2k-1}$. For $1\leq k\leq2g-3$, choose an oriented simple path $\ell_{2k,q_0}\subset\mathbb D$ joining $p_{2k-1,q_0}$ to $p_{2k+1,q_0}$ such that these paths are mutually disjoint away from their endpoints. We further require that $\ell_{2k,q_0}$ meet $\ell_{2k-1,q_0}$ and $\ell_{2k+1,q_0}$ transversely once and be disjoint from all other odd-indexed loops. Their configuration is shown in the upper half of Figure~\ref{fig:Bab-chain-paths}.

\begin{figure}[t!]
\centering
\includegraphics[width=0.8\linewidth]{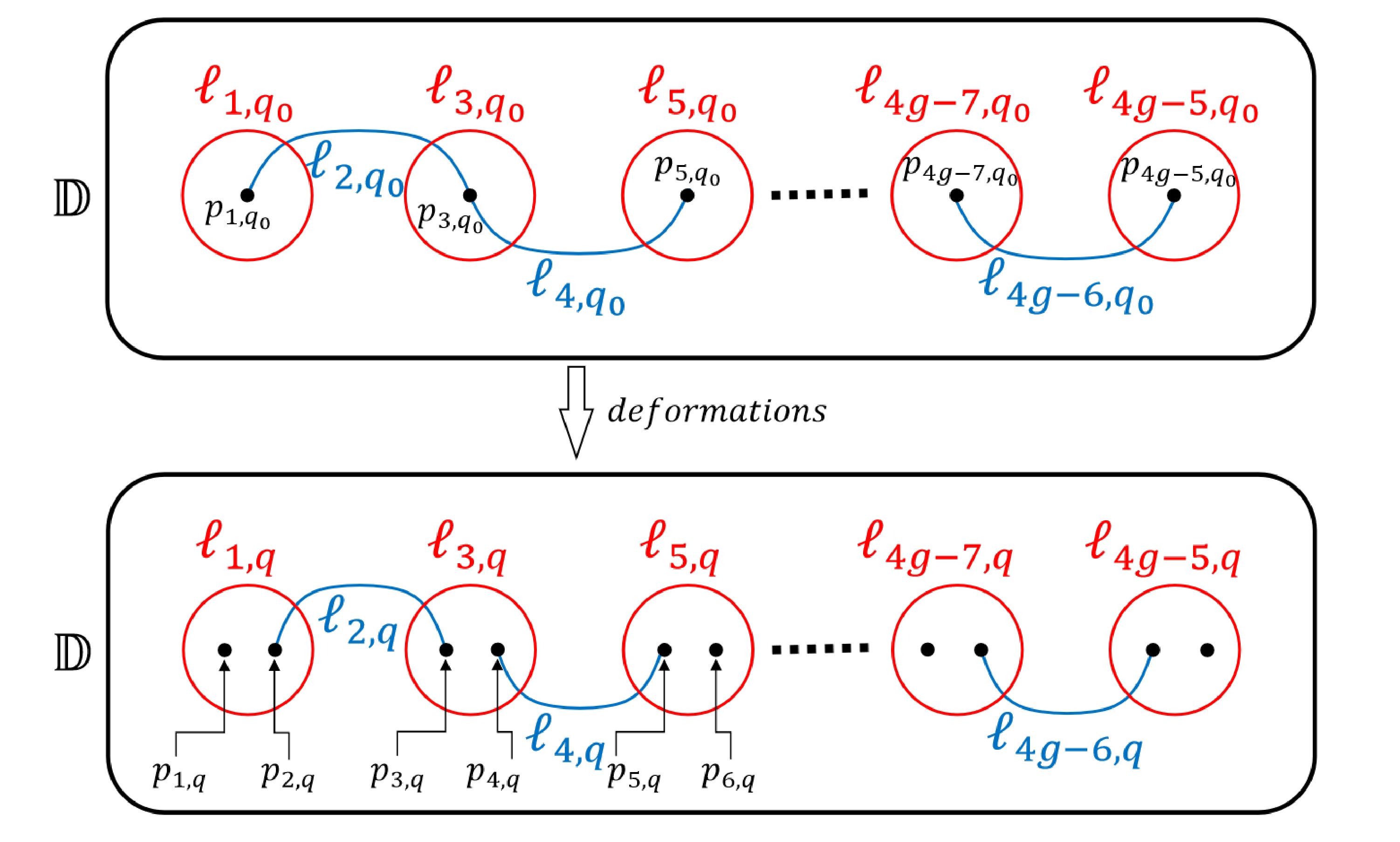}
\figalttext[Chain of auxiliary paths before and after deformation]
{Two stacked schematics inside $\mathbb D$. Above, a chain of auxiliary paths passes through small discs centered at double zeros. Below, each double zero has split into two simple zeros, and adjacent paths terminate at distinct zeros of each pair.}
\caption{The auxiliary paths $\{\ell_{k,q_0}\}_{k=1}^{4g-5}$ and their deformations inside $\mathbb D$. The upper half shows the configuration of $\{\ell_{k,q_0}\}_{k=1}^{4g-5}$, with each disc centered at a double zero of $q_0$. The lower half illustrates the case $q\in\mathcal U\cap\mathcal B^{\mathrm{reg}}$, where each double zero of $q_0$ splits into two simple zeros; we choose distinct endpoints for $\ell_{2k,q}$ and $\ell_{2k+2,q}$.}
\label{fig:Bab-chain-paths}
\end{figure}

Let $a_{g+k,q_0}$ be the lift of $\ell_{2k-1,q_0}$ to $C_1$, and set $\mathfrak a_{g+k,q_0}\coloneqq\sum_{j=1}^{k}a_{g+j,q_0}$ for $1\leq k\leq2g-2$. For $1\leq k\leq2g-3$, denote the lifts of $\ell_{2k,q_0}$ to $C_1$ and $C_2$ by $b_{g+k,q_0}^{(1)}$ and $b_{g+k,q_0}^{(2)}$, respectively, and define $\mathfrak b_{g+k,q_0}\coloneqq b_{g+k,q_0}^{(1)}-b_{g+k,q_0}^{(2)}$. After choosing the orientations of $\mathfrak a_{g+k,q_0}$ appropriately, we have
\[
\langle\mathfrak a_{i,q_0},\mathfrak b_{j,q_0}\rangle=\delta_{ij},\qquad \langle\mathfrak a_{i,q_0},\mathfrak a_{j,q_0}\rangle= \langle\mathfrak b_{i,q_0},\mathfrak b_{j,q_0}\rangle=0,\quad 1\leq i,j\leq3g-3.
\]
After shrinking $\mathcal U$ if necessary, we may assume that each $\mathbb B_{2k-1}$ contains either one double zero or two simple zeros of every $q\in\mathcal U$. We deform the paths $\{\mathfrak a_{i,q_0},\mathfrak b_{i,q_0}\}_{i=1}^{3g-3}$ to $\widetilde\Sigma_q$ as follows. For $1\leq i\leq g$, let $\left\{a_{i,q}^{(1)},a_{i,q}^{(2)}\right\}$ and $\left\{b_{i,q}^{(1)},b_{i,q}^{(2)}\right\}$ be the two lifts of $\alpha_i$ and $\beta_i$, respectively, and define $\mathfrak a_{i,q}$ and $\mathfrak b_{i,q}$ by \eqref{eq:Bab-tangential-cycles}.

For $1\leq k\leq2g-2$, keep $\ell_{2k-1,q}=\ell_{2k-1,q_0}=\partial\mathbb B_{2k-1}$ fixed. Let $a_{g+k,q}$ be the lift of $\ell_{2k-1,q}$ varying continuously from $a_{g+k,q_0}$, and set $\mathfrak a_{g+k,q}\coloneqq\sum_{j=1}^{k}a_{g+j,q}$.

For $1\leq k\leq2g-3$, the endpoints of $\ell_{2k,q_0}$ are double zeros of $q_0$. We deform its two ends according to Type~2 inside $\mathbb B_{2k-1}$ and $\mathbb B_{2k+1}$, as described in Subsection~\ref{subsec:Bd-deformation}, while keeping the path fixed outside these two discs. For $2\leq k\leq2g-3$, when $\mathbb B_{2k-1}$ contains two simple zeros, we choose the endpoints of $\ell_{2k-2,q}$ and $\ell_{2k,q}$ lying in $\mathbb B_{2k-1}$ to be distinct. The resulting deformation is shown in the lower half of Figure~\ref{fig:Bab-chain-paths}.

Let $b_{g+k,q}^{(1)}$ and $b_{g+k,q}^{(2)}$ be the two lifts of $\ell_{2k,q}$, and set $\mathfrak b_{g+k,q}\coloneqq b_{g+k,q}^{(1)}-b_{g+k,q}^{(2)}$. The Type~2 deformation is not unique and different choices may produce paths $\mathfrak b_{g+k,q}$ and $\mathfrak b'_{g+k,q}$ satisfying
\begin{equation}\label{eq:Bab-b-multivalued}
\left[\mathfrak b'_{g+k,q}-\mathfrak b_{g+k,q}\right] = m_k\left[a_{g+k,q}\right]+n_k\left[a_{g+k+1,q}\right], \qquad m_k,n_k\in\mathbb Z.
\end{equation}

As in Lemma~\ref{lem:Bd-deformation}, for each $q\in\mathcal U\cap\mathcal B^{\mathrm{reg}}$, there exists a neighborhood $\mathcal V\ni q$ on which the paths can be chosen continuously so that $\{\mathfrak a_{i,q'},\mathfrak b_{i,q'}\}_{i=1}^{3g-3}$ forms a symplectic basis of $H_1\left(\widetilde\Sigma_{q'},\mathbb R\right)^-$ for every $q'\in\mathcal V$.

Using the deformations constructed above, we define the following functions on $\mathcal U$:
\begin{equation}\label{eq:Bab-period-functions}
\mathfrak y^{g+k}(q)\coloneqq\int_{a_{g+k,q}}\vartheta_q,\qquad \mathfrak z^i(q)\coloneqq\int_{\mathfrak a_{i,q}}\vartheta_q,\qquad \mathfrak w_i(q)\coloneqq-\int_{\mathfrak b_{i,q}}\vartheta_q,
\end{equation}
where $1\leq k\leq2g-2$ and $1\leq i\leq3g-3$. After shrinking $\mathcal U$ if necessary, the functions $\{\mathfrak y^{g+k}\}_{k=1}^{2g-2}$, $\{\mathfrak z^i\}_{i=1}^{3g-3}$, and $\{\mathfrak w_i\}_{i=1}^{g}$ are single-valued and holomorphic on $\mathcal U$. By contrast, $\{\mathfrak w_{g+k}\}_{k=1}^{2g-3}$ are locally holomorphic but generally multivalued on $\mathcal U\cap\mathcal B^{\mathrm{reg}}$. By \eqref{eq:Bab-b-multivalued}, any two local branches satisfy
\begin{equation}\label{eq:Bab-w-branches}
\mathfrak w_{g+k}=\mathfrak w'_{g+k}+m_k\mathfrak y^{g+k}+n_k\mathfrak y^{g+k+1},\qquad m_k,n_k\in\mathbb Z.
\end{equation}

\subsection{Singular model of the special K\"ahler metric near $\mathcal B_{\mathrm{ab}}$}\label{subsec:Bab-singular-model}

We now use the integral functions defined in \eqref{eq:Bab-period-functions} to characterize $\mathcal B_{\mathrm{ab}}$ locally and determine the singular behavior of $\omega_{\mathrm{SK}}$ on $\mathcal U\cap\mathcal B^{\mathrm{reg}}$. We begin with the relations between the functions $\mathfrak y^{g+k}$ and $\mathfrak z^{g+k}$.

\begin{lemma}\label{lem:Bab-period-relations}
For every $q\in\mathcal U$, we have
\[
\sum_{j=1}^{2g-2}\mathfrak y^{g+j}(q)=0, \qquad \mathfrak z^{g+k}(q) = \sum_{j=1}^{k}\mathfrak y^{g+j}(q), \quad 1\leq k\leq2g-3.
\]
\end{lemma}

\begin{proof}
The chain $\sum_{k=1}^{2g-2}a_{g+k,q}$ forms, up to orientation, the boundary of one sheet of $\widetilde\pi_q^{-1}(C\setminus\mathbb D)$. Since $\vartheta_q$ is holomorphic, Stokes' theorem gives $\sum_{k=1}^{2g-2}\mathfrak y^{g+k}(q) = \int_{\sum_{k=1}^{2g-2}a_{g+k,q}}\vartheta_q = 0$. The second identity follows directly from $\mathfrak a_{g+k,q}=\sum_{j=1}^{k}a_{g+j,q}$.
\end{proof}

\begin{theorem}\label{thm:Bab-zero-locus}
After shrinking $\mathcal U$ if necessary,
\[
\mathcal B_{\mathrm{ab}}\cap\mathcal U = \bigl\{ q\in\mathcal U\mid\mathfrak z^{g+k}(q)=0,\ 1\leq k\leq2g-3\bigr\}.
\]
Consequently, $\mathcal B_{\mathrm{ab}}\cap\mathcal U$ is a complex submanifold of dimension $g$, and, under the identification $T_{q_0}\mathcal B\cong H^0(C,K_C^2)$, we have $T_{q_0}\mathcal B_{\mathrm{ab}}=H^0\bigl(C,K_C^2(-D_{q_0})\bigr)$.
\end{theorem}

\begin{proof}
By Lemma~\ref{lem:quadratic-vanishing}, $\mathfrak y^{g+k}(q)=0$ if and only if $q$ has a double zero in $\mathbb B_{2k-1}$. After shrinking $\mathcal U$ so that $\mathcal U\cap\mathcal B_{2g-2}=\varnothing$, the common zero locus of $\{\mathfrak y^{g+k}\}_{k=1}^{2g-2}$ consists precisely of $\mathcal B_{\mathrm{ab}}\cap\mathcal U$. By Lemma~\ref{lem:Bab-period-relations}, it coincides with the common zero locus of $\mathfrak z^{g+1},\ldots,\mathfrak z^{3g-3}$.

By Lemma~\ref{lem:Bd-vanishing-differential}, the common kernel of $\mathrm d\mathfrak y^{g+1},\ldots,\mathrm d\mathfrak y^{3g-2}$ at $q_0$ is $H^0\bigl(C,K_C^2(-D_{q_0})\bigr)$. Since $D_{q_0}=\operatorname{div}(\psi)$, we have $K_C^2(-D_{q_0})\cong K_C$, and hence this kernel has dimension $g$. Thus $\{\mathrm d\mathfrak y^{g+k}\}_{k=1}^{2g-2}$ has rank $2g-3$. The relations in Lemma~\ref{lem:Bab-period-relations} show that $\mathrm d\mathfrak z^{g+1},\ldots,\mathrm d\mathfrak z^{3g-3}$ have the same common kernel and are therefore linearly independent at $q_0$. The implicit function theorem gives the asserted submanifold structure.
\end{proof}

To express the coefficients of $\omega_{\mathrm{SK}}$, we next introduce the meromorphic differentials normalized by the adapted $\mathfrak a$-paths. Let $\widetilde p_{2k-1,q_0}^{(1)}\in C_1$ and $\widetilde p_{2k-1,q_0}^{(2)}\in C_2$ denote the two preimages of $p_{2k-1,q_0}$.

\begin{lemma}\label{lem:Bab-dual-forms}
For every $q\in\mathcal U$, there exists a unique basis $\{\omega_{i,q}\}_{i=1}^{3g-3}$ of $H^0\bigl(\widetilde\Sigma_q,K_{\widetilde\Sigma_q}(\widetilde D_q)\bigr)^-$ satisfying
\[
\int_{\mathfrak a_{i,q}}\omega_{j,q}=\delta_{ij},\qquad 1\leq i,j\leq3g-3.
\]
At $q=q_0$, the forms $\omega_{i,q_0}$ are holomorphic for $1\leq i\leq g$. For each $1\leq k\leq2g-3$, the restriction of $\omega_{g+k,q_0}$ to $C_1$ has simple poles only at $\widetilde p_{2k-1,q_0}^{(1)}$ and $\widetilde p_{2k+1,q_0}^{(1)}$, with residues $\pm(2\pi\mathrm i)^{-1}$, while its restriction to $C_2$ has the opposite residues at the corresponding points.
\end{lemma}

\begin{proof}
For $q\in\mathcal U\cap\mathcal B_{\mathrm{ab}}$, $\widetilde\Sigma_q=C_1\sqcup C_2$. The paths $\{\mathfrak a_{i,q},\mathfrak b_{i,q}\}_{i=1}^{g}$ form a symplectic basis of $H_1(\widetilde\Sigma_q,\mathbb R)^-$, while $\mathfrak a_{g+k,q}\subset C_1$ is a contractible loop enclosing $\big\{\widetilde p_{2j-1,q}^{(1)}\big\}_{j=1}^{k}$, and these loops are nested. Corollary~\ref{cor:nested-normalized-differentials} therefore gives the required basis, with the residues on $C_2$ determined by anti-invariance.

Suppose now that $q\in\mathcal U\setminus\mathcal B_{\mathrm{ab}}$, and shrink $\mathcal U$ so that $\mathcal U\cap\mathcal B_{2g-2}=\varnothing$. Let $I_q$ denote the set of indices $k\in\{1,\ldots,2g-2\}$ for which $q$ has a double zero in $\mathbb B_{2k-1}$. Then $\widetilde\Sigma_q$ is connected. For $k\in I_q$, the loop $a_{g+k,q}$ is contractible, while for $k\notin I_q$ it is noncontractible. Since $q\notin\mathcal B_{2g-2}$, we have $I_q\neq\{1,\ldots,2g-2\}$. Assume first that $2g-2\notin I_q$. Then the noncontractible loops among $\{\mathfrak a_{i,q}\}_{i=1}^{g}\cup\{a_{g+k,q}\}_{k=1}^{2g-3}$ are linearly independent and can be completed to the $\mathfrak a$-part of a symplectic basis of $H_1(\widetilde\Sigma_q,\mathbb R)^-$, while the remaining $a_{g+k,q}$ are contractible loops around the preimages of the nodes. Hence Lemma~\ref{lem:adapted-normalized-differentials} gives a unique basis $\{\eta_{i,q}\}_{i=1}^{3g-3}$ of $H^0(\widetilde\Sigma_q,K_{\widetilde\Sigma_q}(\widetilde D_q))^-$ dual to these normalization paths. 

If $2g-2\in I_q$, choose $r\notin I_q$ and repeat the argument above to obtain $1$-forms dual to
$\{\mathfrak a_{i,q}\}_{i=1}^{g}\cup\{a_{g+k,q}\}_{k=1,\ k\neq r}^{2g-2}$. Since, up to orientation, $\sum_{k=1}^{2g-2}a_{g+k,q}$ is the boundary of one sheet of $\widetilde\pi_q^{-1}(C\setminus\mathbb D)$, on which these differentials are holomorphic, Stokes' theorem gives $\sum_{k=1}^{2g-2}\int_{a_{g+k,q}}\eta=0$. Thus an invertible linear transformation yields the $1$-forms dual to $\{\mathfrak a_{i,q}\}_{i=1}^{g}\cup\{a_{g+k,q}\}_{k=1}^{2g-3}$.

Finally, since $\mathfrak a_{g+k,q}=\sum_{j=1}^{k}a_{g+j,q}$ for $1\leq k\leq2g-3$, we set $\omega_{i,q}\coloneqq\eta_{i,q}$ for $1\leq i\leq g$, $\omega_{g+k,q}\coloneqq\eta_{g+k,q}-\eta_{g+k+1,q}$ for $1\leq k\leq2g-4$, and $\omega_{3g-3,q}\coloneqq\eta_{3g-3,q}$, which satisfies the required conditions.
\end{proof}

Via the isomorphism $\Psi_q$ in \eqref{eq:H0-isomorphism}, the differentials $\omega_{i,q}$ depend continuously on $q\in\mathcal U$ by the same matrix argument as in Lemma~\ref{lem:family-normalization}. On any simply connected open subset of $\mathcal U\cap\mathcal B^{\mathrm{reg}}$, choose branches of $\{\mathfrak w_{g+k}\}_{k=1}^{2g-3}$. Write the coefficient matrix in block form
\begin{equation}\label{eq:Bab-SK-metric}
\omega_{\mathrm{SK}} = \frac{\mathrm i}{2} \sum_{i,j=1}^{3g-3} \operatorname{Im}(\tau_{ij})\, \mathrm d\mathfrak z^i\wedge\mathrm d\overline{\mathfrak z^j}, \quad \tau_{ij}(q) \coloneqq -\frac{\partial\mathfrak w_j}{\partial\mathfrak z^i}(q) = \int_{\mathfrak b_{j,q}}\omega_{i,q}; \quad \bigl(\operatorname{Im}(\tau_{ij})\bigr)_{i,j=1}^{3g-3} \coloneqq
\begin{pmatrix}
A & B\\
B^{\mathsf T} & D
\end{pmatrix},
\end{equation}
where $A$ and $D$ have orders $g$ and $2g-3$, respectively. By \eqref{eq:Bab-w-branches} and Lemma~\ref{lem:Bab-period-relations}, $\operatorname{Im}(\tau_{ij})$ is independent of the choices of $\mathfrak b$-paths. The asymptotic behavior of the coefficient matrix is described as follows.

\begin{table*}[!t]
\caption{Divergence pattern of the block $D(q)$ as $q\to q_0$.\label{tab:Bab-divergence-pattern}}
\tabcolsep=0pt
\begin{tabular*}{\textwidth}{@{\extracolsep{\fill}}lccccc@{\extracolsep{\fill}}}
\toprule
Integration path & $\omega_{g+1,q}$ & $\omega_{g+2,q}$ & $\omega_{g+3,q}$ & $\cdots$ & $\omega_{3g-3,q}$\\
\midrule
$\mathfrak b_{g+1,q}$ & $L_{g+1}+L_{g+2}$ & $-L_{g+2}$ & $c$ & $\cdots$ & $c$\\
$\mathfrak b_{g+2,q}$ & $-L_{g+2}$ & $L_{g+2}+L_{g+3}$ & $-L_{g+3}$ & $\cdots$ & $c$\\
$\mathfrak b_{g+3,q}$ & $c$ & $-L_{g+3}$ & $L_{g+3}+L_{g+4}$ & $\cdots$ & $c$\\
$\vdots$ & $\vdots$ & $\vdots$ & $\vdots$ & $\ddots$ & $\vdots$\\
$\mathfrak b_{3g-3,q}$ & $c$ & $c$ & $c$ & $\cdots$ & $L_{3g-3}+L_{3g-2}$\\
\botrule
\end{tabular*}
\begin{tablenotes}
\item Note: Each displayed logarithmic expression records uniform two-sided comparability in the sense of $\asymp$ and each occurrence of $c$ denotes a possibly different entry that extends continuously to $q_0$.
\end{tablenotes}
\end{table*}

\begin{theorem}\label{thm:Bab-singular-model}

For $1\leq k\leq2g-2$, set $L_{g+k}(q)\coloneqq-\log|\mathfrak y^{g+k}(q)|$. As $\mathcal B^{\mathrm{reg}}\ni q\to q_0$, the blocks $A(q)$ and $B(q)$ extend continuously to $q_0$, while only the diagonal and first off-diagonal entries of $D(q)$ may diverge. More explicitly, the following uniform two-sided logarithmic estimates hold:
\begin{align*}
&\textnormal{(1)}\quad D_{kk}(q) \asymp L_{g+k}(q)+L_{g+k+1}(q), &&1\leq k\leq2g-3,\\
&\textnormal{(2)}\quad -D_{k,k+1}(q)=-D_{k+1,k}(q) \asymp L_{g+k+1}(q), &&1\leq k\leq2g-4.
\end{align*}
Moreover, $D_{ij}(q)$ extends continuously to $q_0$ whenever $|i-j|>1$. The resulting divergence pattern is summarized in Table~\ref{tab:Bab-divergence-pattern}.
\end{theorem}

\begin{proof}
By the representation \eqref{eq:Bab-SK-metric}, if $i\leq g$, then $\omega_{i,q_0}$ is holomorphic. If $i=g+k$ and $j=g+l$ with $|k-l|>1$, then $\mathfrak b_{g+l,q_0}$ does not meet the poles of $\omega_{g+k,q_0}$. Hence the entries of $A(q)$, $B(q)$, and $D_{kl}(q)$ with $|k-l|>1$ extend continuously to $q_0$. The asymptotic behavior of the remaining entries of $D(q)$ is described as follows.

For $1\leq k\leq2g-3$, we have
\[
D_{kk}(q)=\operatorname{Im}\left(\int_{\mathfrak b_{g+k,q}}\omega_{g+k,q}\right).
\]
The singular contributions to $D_{kk}(q)$ arise from the endpoints of $\mathfrak b_{g+k,q_0}$ lying over $p_{2k-1,q_0}$ and $p_{2k+1,q_0}$. By Lemmas~\ref{lem:local-logarithmic-estimate} and \ref{lem:quadratic-vanishing}, the contributions inside $\mathbb B_{2k-1}$ and $\mathbb B_{2k+1}$ are uniformly comparable to $L_{g+k}(q)$ and $L_{g+k+1}(q)$, respectively. Taking the orientation into account, the two singular models add up, proving (1).

For $1\leq k\leq2g-4$, we have
\[
D_{k,k+1}(q)=\operatorname{Im}\left(\int_{\mathfrak b_{g+k+1,q}}\omega_{g+k,q}\right).
\]
The singular contributions to $D_{k,k+1}(q)$ arise from the endpoints of $\mathfrak b_{g+k+1,q_0}$ lying over $p_{2k+1,q_0}$. The same local calculation, together with the orientation and residue at these poles, gives $-D_{k,k+1}(q)\asymp L_{g+k+1}(q)$. Since $(\tau_{ij})$ is symmetric, the same estimate holds for $D_{k+1,k}(q)$, proving (2).
\end{proof}

\begin{remark}
This asymptotic behavior depends on the chosen conjugate special coordinates. Unlike the $\mathcal B_d$ case, where divergences appear only on the diagonal, here additional divergent terms occur in the off-diagonal entries. A natural concern is whether one can design local coordinates that yield a simpler, more intrinsic description of the metric singularity. This is not merely a problem of diagonalizing a symmetric matrix via linear algebra. A key constraint is that coordinates cannot be chosen arbitrarily; we must work with conjugate special coordinate systems. Freed \cite[Proposition~1.38(d)]{freed1999special} showed that the transition map between any two such systems is an affine transformation whose linear part is a constant $\operatorname{Sp}(2n,\mathbb R)$ matrix. This restriction substantially limits the possible simplifications of the singular block $D$.
\end{remark}

\section{Tangential limits and radial flat cones}\label{sec:radial}

In Sections~\ref{sec:Bd}, \ref{sec:B2g-2}, and \ref{sec:Bab}, we determined the singular behavior of the special K\"ahler metric near the strata $\mathcal B_d$, $\mathcal B_{2g-2}$, and $\mathcal B_{\mathrm{ab}}$, respectively. We now derive two consequences of these local models. First, although $\omega_{\mathrm{SK}}$ diverges in the normal directions, its tangential restriction converges to the intrinsic special K\"ahler metric on each stratum. We then restrict this intrinsic metric to complex lines through the origin and recover a flat cone metric. It is worth noting that Hitchin \cite[Proposition~2]{hitchin2021integrable} gave a K\"ahler potential from which the cone formula \eqref{eq:radial-flat-cone} also follows directly by differentiation. Nevertheless, we retain the period-coordinate description below in order to relate the radial geometry to the preceding local models.

Let $\mathcal S$ denote one of the strata $\mathcal B_d$, $1\leq d\leq2g-3$, $\mathcal B_{2g-2}$, or $\mathcal B_{\mathrm{ab}}$, and fix $q_0\in\mathcal S$. Set
\[
B_{\mathcal S},\ \text{and}\ B_{\mathcal S}^{\mathrm c}\coloneqq
\begin{cases}
\{1,\ldots,g\}\cup\{g+d+1,\ldots,3g-3\},\quad\text{and}\quad \{g+1,\ldots,g+d\},& \text{if }\mathcal S=\mathcal B_d,\\
\{2,\ldots,g\},\quad\text{and}\quad \{g+1,\ldots,3g-2\},& \text{if }\mathcal S=\mathcal B_{2g-2},\\
\{1,\ldots,g\},\quad\text{and}\quad\{g+1,\ldots,3g-3\},& \text{if }\mathcal S=\mathcal B_{\mathrm{ab}}.
\end{cases}
\]
Each stratum $\mathcal S$ carries the intrinsic special K\"ahler structure described above; denote its metric by $\omega_{\mathrm{SK},\mathcal S}$. The adapted paths $\{\mathfrak a_{i,q},\mathfrak b_{i,q}\}_{i\in B_{\mathcal S}}$ form a symplectic basis of $H_1(\widetilde\Sigma_q,\mathbb R)^-$ for $q\in\mathcal S\cap\mathcal U$. Hence $\{\mathfrak z^i,\mathfrak w_i\}_{i\in B_{\mathcal S}}$ are conjugate special coordinates on $\mathcal S\cap\mathcal U$; see \cite[Lemma~6.1]{he2025asymptotics} for more details. We have
\begin{equation}\label{eq:stratum-metric}
\omega_{\mathrm{SK},\mathcal S}=\frac{\mathrm i}{2}\sum_{i,j\in B_{\mathcal S}}\operatorname{Im}(\tau_{ij})\,\mathrm d\mathfrak z^i\wedge\mathrm d\overline{\mathfrak z}^j,\qquad \tau_{ij}\coloneqq-\frac{\partial\mathfrak w_j}{\partial\mathfrak z^i}.
\end{equation}
The following lemma describes the tangential asymptotic behavior of $\omega_{\mathrm{SK}}$.

\begin{lemma}\label{lem:tangential-metric-limit}
For every $q\in\mathcal U$, define $E_q^{\mathcal S}\coloneqq\bigcap_{k\in B_{\mathcal S}^{\mathrm c}}\ker\left(\mathrm d\mathfrak z^k|_q\right)\subset T_q\mathcal B$. After shrinking $\mathcal U$ if necessary, $E_q^{\mathcal S}=T_q\mathcal S$ for every $q\in\mathcal S\cap\mathcal U$, and
\[
\omega_{\mathrm{SK},\mathcal S}(q_0) = \lim_{\mathcal B^{\mathrm{reg}}\ni q\to q_0}\left(\left.\omega_{\mathrm{SK}}(q)\right|_{E_q^{\mathcal S}}\right).
\]
\end{lemma}

\begin{proof}
By Theorems~\ref{thm:Bd-zero-locus}, \ref{thm:B2g-2-zero-locus}, and \ref{thm:Bab-zero-locus}, the stratum $\mathcal S\cap\mathcal U$ is locally defined by $\mathfrak z^k=0$ for $k\in B_{\mathcal S}^c$. Therefore, $E_q^{\mathcal S}=T_q\mathcal S$ for $q\in\mathcal S\cap\mathcal U$. The metric $\omega_{\mathrm{SK}}$ is expressed on $\mathcal U\cap\mathcal B^{\mathrm{reg}}$ as
\[
\omega_{\mathrm{SK}}=\frac{\mathrm i}{2}\sum_{i,j\in B_{\mathcal S}\cup B_{\mathcal S}^c}\operatorname{Im}(\tau_{ij})\,\mathrm d\mathfrak z^i\wedge\mathrm d\overline{\mathfrak z}^j,\qquad \tau_{ij}\coloneqq-\frac{\partial\mathfrak w_j}{\partial\mathfrak z^i}.
\]
Restricting this expression to $E_q^{\mathcal S}$ leaves only the terms with $i,j\in B_{\mathcal S}$. By Theorems~\ref{thm:Bd-singular-model}, \ref{thm:B2g-2-singular-model}, and \ref{thm:Bab-singular-model}, the corresponding coefficients $\operatorname{Im}(\tau_{ij})$ extend continuously to $q_0$ and the limits are precisely the coefficients of $\omega_{\mathrm{SK},\mathcal S}(q_0)$ in \eqref{eq:stratum-metric}. This proves the result.
\end{proof}

\begin{remark}
For $q\in\mathcal S\cap\mathcal U$, the adapted paths $\{\mathfrak a_{i,q},\mathfrak b_{i,q}\}_{i\in B_{\mathcal S}}$ form a symplectic basis of $H_1(\widetilde\Sigma_q,\mathbb R)^-$. Hence $\{\mathfrak z^i,\mathfrak w_i\}_{i\in B_{\mathcal S}}$ are conjugate special coordinates on $\mathcal S\cap\mathcal U$, and $\{\mathfrak z^i|_{\mathcal S}\}_{i\in B_{\mathcal S}}$ are local holomorphic coordinates. The corresponding theorems \ref{thm:Bd-zero-locus}, \ref{thm:B2g-2-zero-locus}, and \ref{thm:Bab-zero-locus} show that the differentials $\{\mathrm d\mathfrak z^k|_{q_0}\}_{k\in B_{\mathcal S}^c}$ are linearly independent with common kernel $T_{q_0}\mathcal S$. Consequently, the full collection $\{\mathrm d\mathfrak z^i|_{q_0}\}_{i\in B_{\mathcal S}\cup B_{\mathcal S}^c}$ is linearly independent, and, after shrinking $\mathcal U$, the corresponding functions form holomorphic coordinates on $\mathcal U$.

Moreover, for $i,j\in B_{\mathcal S}$, the matrix $(\tau_{ij}(q_0))$ is precisely the period matrix of the intrinsic integrable system over $\mathcal S$. Hence the limiting tangential block $A(q_0)$ in Theorem~\ref{thm:main-asymptotic} is the coefficient matrix of $\omega_{\mathrm{SK},\mathcal S}(q_0)$ in the local coordinates $\{\mathfrak z^i\}_{i\in B_{\mathcal S}}$, and is therefore positive definite.
\end{remark}

We next study the intrinsic special K\"ahler metric on $\mathcal S$ along the complex line through $q_0$ and the origin. Set
\[
\mathcal L_{q_0}\coloneqq\{\ell q_0\mid\ell\in\mathbb C\},\qquad\mathcal L_{q_0}^{\times}\coloneqq\mathcal L_{q_0}\setminus\{0\}.
\]
Since nonzero scaling preserves the stratum, $\mathcal L_{q_0}^{\times}\subset\mathcal S$. Denote the restriction of $\omega_{\mathrm{SK},\mathcal S}$ to $\mathcal L_{q_0}^{\times}$ by $\omega_{\mathcal L_{q_0}}$. By Lemma~\ref{lem:tangential-metric-limit}, $\omega_{\mathcal L_{q_0}}$ may be viewed at each point of $\mathcal L_{q_0}^{\times}$ as the tangential limit of $\omega_{\mathrm{SK}}$. Along this punctured line, the normalized spectral curves are related by scaling. We record this scaling behavior in order to express the radial metric in the same period-coordinate framework used above.

Let $V\subset\mathbb C^*$ be a simply connected open subset and fix a holomorphic branch of $\sqrt{\ell}$ on $V$. For $\ell\in V$, this branch determines a biholomorphism
\[
F_\ell\colon\widetilde\Sigma_{q_0}\longrightarrow\widetilde\Sigma_{\ell q_0},\qquad(\widetilde x,\widetilde z)\longmapsto(\sqrt{\ell}\,\widetilde x,\widetilde z),
\]
where $z$ is a local coordinate on $C$ and $(\widetilde x,\widetilde z)$ is the induced coordinate on $\operatorname{Tot}(K_C)$. Moreover, $F_\ell^*\vartheta_{\ell q_0}=\sqrt{\ell}\,\vartheta_{q_0}$. Given adapted paths $\{\mathfrak a_{i,q_0},\mathfrak b_{i,q_0}\}_{i\in B_{\mathcal S}}$ on $\widetilde\Sigma_{q_0}$, define the corresponding paths on $\widetilde\Sigma_{\ell q_0}$ by
\[
\mathfrak a_{i,\ell q_0}\coloneqq(F_\ell)_*\mathfrak a_{i,q_0},\qquad\mathfrak b_{i,\ell q_0}\coloneqq(F_\ell)_*\mathfrak b_{i,q_0},\qquad i\in B_{\mathcal S}.
\]

\begin{lemma}[Scaling of periods]\label{lem:radial-scaling}
With respect to the fixed branch of $\sqrt{\ell}$ on $V$, we have
\[
\mathfrak z^i(\ell q_0)=\sqrt{\ell}\,\mathfrak z^i(q_0),\quad\mathfrak w_i(\ell q_0)=\sqrt{\ell}\,\mathfrak w_i(q_0),\quad F_\ell^*\omega_{i,\ell q_0}=\omega_{i,q_0},\quad\text{for }\ell\in V,\ i\in B_{\mathcal S}.
\]
Consequently, $\tau_{ij}(\ell q_0)=\tau_{ij}(q_0)$ for $i,j\in B_{\mathcal S}$.
\end{lemma}

\begin{proof}
Since $F_\ell^*\vartheta_{\ell q_0}=\sqrt{\ell}\,\vartheta_{q_0}$, we have
\[
\mathfrak z^i(\ell q_0) = \int_{\mathfrak a_{i,\ell q_0}}\vartheta_{\ell q_0} = \int_{\mathfrak a_{i,q_0}}F_\ell^*\vartheta_{\ell q_0} =\sqrt{\ell}\int_{\mathfrak a_{i,q_0}}\vartheta_{q_0} = \sqrt{\ell}\,\mathfrak z^i(q_0),
\]
and similarly $\mathfrak w_i(\ell q_0)=\sqrt{\ell}\,\mathfrak w_i(q_0)$. Moreover,
\[
\int_{\mathfrak a_{j,q_0}}F_\ell^*\omega_{i,\ell q_0} = \int_{\mathfrak a_{j,\ell q_0}}\omega_{i,\ell q_0} = \delta_{ij}.
\]
The uniqueness of the normalized differentials therefore gives $F_\ell^*\omega_{i,\ell q_0}=\omega_{i,q_0}$. Consequently,
\[
\tau_{ij}(\ell q_0) = \int_{\mathfrak b_{j,\ell q_0}}\omega_{i,\ell q_0} = \int_{\mathfrak b_{j,q_0}}F_\ell^*\omega_{i,\ell q_0} = \tau_{ij}(q_0).
\]
This completes the proof.
\end{proof}

Replacing the chosen branch of $\sqrt{\ell}$ by its negative amounts to composing $F_\ell$ with the deck transformation. Accordingly, $\mathfrak z^i$ and $\mathfrak w_i$ both change sign, whereas $\tau_{ij}$ and $\omega_{\mathcal L_{q_0}}$ remain unchanged. Therefore, the resulting metric on $\mathcal L_{q_0}^{\times}$ is independent of the local choice of branch.

\begin{corollary}[Radial flat cone]\label{thm:radial-flat-cone}
For every $q_0\in\mathcal S$, the restriction of $\omega_{\mathrm{SK},\mathcal S}$ to $\mathcal L_{q_0}^{\times}$ is flat and extends across the origin as a flat cone metric with cone angle $\pi$:
\begin{equation}\label{eq:radial-flat-cone}
\omega_{\mathcal L_{q_0}}=\frac{\mathrm i}{2}C(q_0)|\ell|^{-1}\,\mathrm d\ell\wedge\mathrm d\overline{\ell},\qquad C(q_0)=\frac{\mathrm i}{4}\int_C\sqrt{\big.q_0\overline{q_0}}.
\end{equation}
\end{corollary}

\begin{proof}
By Lemma~\ref{lem:radial-scaling}, the entries $\tau_{ij}$ are constant along $\mathcal L_{q_0}^{\times}$, while $\mathfrak z^i(\ell q_0)=\sqrt{\ell}\,\mathfrak z^i(q_0)$. Hence, on $\mathcal L_{q_0}^{\times}$,
\[
\mathrm d\mathfrak z^i\wedge\mathrm d\overline{\mathfrak z^j} = \frac{\mathfrak z^i(q_0)\overline{\mathfrak z^j(q_0)}}{4|\ell|}
\,\mathrm d\ell\wedge\mathrm d\overline{\ell}.
\]
Restricting \eqref{eq:stratum-metric} to $\mathcal L_{q_0}^{\times}$ therefore gives $\displaystyle\omega_{\mathcal L_{q_0}} =\frac{\mathrm i}{2}C(q_0)|\ell|^{-1}\,\mathrm d\ell\wedge\mathrm d\overline{\ell}$ for some positive constant $C(q_0)$. Hitchin \cite[Proposition~2]{hitchin2021integrable} gave the following K\"ahler potential on the stratum:
\[
\mathcal K_{\mathcal S}(q) = \frac{\mathrm i}{4} \int_{\widetilde\Sigma_q} \vartheta_q\wedge\overline{\vartheta_q} = \frac{\mathrm i}{2} \int_C\sqrt{q\overline q}.
\]
Thus $\mathcal K_{\mathcal S}(\ell q_0) = \frac{\mathrm i}{2}|\ell|\int_C\sqrt{q_0\overline{q_0}}$. On the other hand, a K\"ahler potential for $\frac{\mathrm i}{2}C(q_0)|\ell|^{-1}\,\mathrm d\ell\wedge\mathrm d\overline{\ell}$ is $2C(q_0)|\ell|$. Comparing these expressions yields
\[
C(q_0)=\frac12\mathcal K_{\mathcal S}(q_0)=\frac{\mathrm i}{4}\int_C\sqrt{q_0\overline{q_0}}.
\]
This completes the proof.
\end{proof}

\begin{remark}
When $g=2$, Hitchin \cite[Section~5.2]{hitchin2021integrable} obtained the same flat cone metric as in Corollary~\ref{thm:radial-flat-cone}.
\end{remark}

The results in this paper also point to several directions for further investigation. One is the construction of hyperK\"ahler local models near singular Hitchin fibers. Fredrickson--Zimet \cite{fredrickson2025hyper} proposed multi-Ooguri--Vafa-type models within the GMN framework, whose local hypotheses require the conjugate coordinates associated with vanishing central charges to exhibit controlled logarithmic behavior and prescribed monodromy. The multivalued integral functions and logarithmic asymptotics of the period matrix obtained in this paper are compatible with this structure. It is therefore natural to ask whether the remaining local hypotheses can be verified for the $\mathrm{SL}_2(\mathbb C)$-Hitchin system near degenerations with nodal spectral curves. Such a verification would place these degenerations within the framework developed in \cite{fredrickson2025hyper} and provide a basis for studying the instanton corrections to the semi-flat metric.

Another direction is to extend the analysis to more general strata of the discriminant locus. In this paper, we consider only strata for which all zeros have multiplicity at most $2$. For a stratum $\mathcal B_{\mathbf p}$ determined by a general partition $\mathbf p$, higher-order singularities of the spectral curves give rise to more complicated vanishing cycles, monodromy transformations, and anti-invariant homology. The smoothness of such strata also requires further analysis, particularly when the length of the partition is less than $2g-2$, where the existing results in \cite{he2025asymptotics} do not apply directly. Although integrable subsystems in the sense of Hitchin may not exist on arbitrary partition strata, the complex torus families constructed by Horn \cite{horn2022semiabelian} suggest that some remnant of the abelian-fibration structure may nevertheless persist.

\paragraph*{\textbf{Funding.}} Zhenxi Huang is supported by the Guangzhou Municipal Science and Technology Bureau, Guangzhou Key Research and Development Program (Grant No.~2024A04J3560). Bin Xu is supported in part by the Project of Stable Support for Youth Team in Basic Research Field, Chinese Academy of Sciences (Grant No.~YSBR-001), and by the National Natural Science Foundation of China (Grant No.~12271495).

\paragraph*{\textbf{Acknowledgments.}} We are deeply grateful to Nigel Hitchin for drawing our attention to his work \cite{hitchin2019critical,hitchin2021integrable}, which provided the conceptual framework for our description of the tangential limits of $\omega_{\mathrm{SK}}$. We thank Erwan Lanneau for his generous help in understanding moduli spaces of Riemann surfaces equipped with quadratic differentials, and David Baraglia for his insights into the existence of trivializing discs, which played an important role in our approach.

We are also grateful to Laura Fredrickson for valuable discussions on the GMN construction, and to Johannes Horn for sharing his insights into Hitchin fibers over the discriminant locus, especially the abelian subvarieties contained in these fibers. The last author thanks Andriy Haydys for introducing him to algebraically completely integrable systems, Xiaoyu Su for helping him understand the BNR correspondence, and Mao Sheng for bringing the recent work of He--Horn--Li \cite{he2025asymptotics} to his attention. We also thank Andrew Neitzke for his lecture notes on Higgs bundle moduli spaces \cite{neitzke2016moduli} and for patiently answering our questions, Nianzi Li for clarifying the smooth structure of $\mathcal B_{2g-2}$, and Sicheng Lu for helpful discussions on symplectic bases of Riemann surfaces.

\paragraph*{\textbf{Conflicts of interest.}} The authors declare that they have no competing interests.

\paragraph*{\textbf{Data availability.}} No data were generated or analysed during the current study.

\bibliographystyle{oup-plain}
\bibliography{Ref}
\end{document}